\documentclass[reqno, 15pt]{amsart}
\usepackage{latexsym}
\usepackage{amsmath}
\usepackage{amssymb}
\usepackage{amsthm}
\usepackage{amscd}
\usepackage{graphicx}
\usepackage{xcolor}
\usepackage[all,cmtip]{xy}
\usepackage[colorlinks,citecolor=blue]{hyperref}
\usepackage{tikz}
\usepackage{tikz-cd}
\usepackage{caption} 
\usepackage[a4paper,top=3cm,bottom=3cm,left=3cm,right=3cm]{geometry}
\usepackage{cite}
\usepackage{enumitem}
\usepackage{tabularx}
\usepackage{multirow}
\usepackage{multicol}
\usepackage{longtable}
\usepackage{makecell}

\usepackage{float} 
\newtheorem{theorem}{Theorem}[section]
\newtheorem{corollary}[theorem]{Corollary}
\newtheorem{lemma}[theorem]{Lemma}
\newtheorem{proposition}[theorem]{Proposition}

\newcounter{claimcnt}
\AtBeginEnvironment{proof}{\setcounter{claimcnt}{0}}
\makeatletter
\newenvironment{claim}
{\refstepcounter{claimcnt}%
	\protected@edef\@currentlabel{\thetheorem.\theclaimcnt}%
	\noindent\textbf{Claim \@currentlabel.}\quad}
{\par}
\makeatother

\newtheorem{prob}[theorem]{Problem}
\theoremstyle{definition}
\newtheorem{definition}[theorem]{Definition}
\newtheorem{ex}[theorem]{Example}
\theoremstyle{remark}
\newtheorem{remark}[theorem]{Remark}

\numberwithin{equation}{section}

\DeclareMathOperator{\oh}{\mathcal{O}}
\begin{document}
	
	\title[]{Uniform non-homogeneous bundles on quadrics}
	
	\author{Xinyi Fang}
	\address{Xinyi Fang, Department of Mathematics, Shanghai Normal University, Shanghai, 200234, PR China}
	\thanks{1.Xinyi Fang is supported by National Natural Science Foundation of China (Grant No. 12501057) and
		National Natural Science Foundation of China (Grant No. 12471040)}
	\email{xinyif@shnu.edu.cn}
	\author{Yuhang Zhou}
	\address{Yuhang Zhou, School of Mathematical Sciences,
		University of Science and Technology of China,
		No.96, JinZhai Road Baohe District, 
		Hefei, Anhui, 230026, PR China}
	\thanks{2.Yuhang Zhou is supported by CAS Project for Young Scientists in Basic Research, Grant No. YSBR-032}
	\email{zyh@ustc.edu.cn}

	\begin{abstract} 
Let $X$ be an $n$-dimensional generalized Grassmannian not isomorphic to $\mathbb{P}^n$.
	We prove that $k(X)\le n-1$, where $k(X)$ denotes the maximal integer such that every uniform bundle on $X$ of rank at most $k(X)$ is homogeneous. In particular, for smooth quadrics $\mathbb{Q}^n$, we have $k(\mathbb{Q}^n)=n-1$ for odd $n$, and $n-2\le k(\mathbb{Q}^n)\le n-1$ for even $n$.
	We classify uniform rank $n$ bundles on $\mathbb{Q}^{n}$ for $n=3$, $5$. Furthermore, we characterize projective spaces among generalized Grassmannians in terms of uniform bundles.
			\end{abstract}
		\maketitle
	

		%We classify uniform $n$-bundles on $\mathbb{Q}^n$ for $n=3$, $5$, and characterize projective spaces among generalized Grassmannians in terms of uniform bundles.
		
		%Let $E$ be a uniform bundle on an arbitrary generalised Grassmannian $X$ defined over $\mathbb{C}$. We show that  if the rank of $E$ is at most $e.d.(\mathrm{VMRT})$, then $E$  necessarily splits. For some generalised Grassmannians, we prove that the upper bounds $e.d.(\mathrm{VMRT})$ are optimal and classify all unsplit uniform bundles of minimal ranks.  Under some special assumptions, we show that  morphisms to some generalised flag varieties must be constant, which partially answered a conjecture of Kumar.

%	\setcounter{tocdepth}{1}\tableofcontents
	
	%=====================================================================
	\section{Introduction}
	
Vector bundles on a projective variety over the complex number field $\mathbb{C}$ are fundamental research objects in algebraic geometry. It is classically known that every vector bundle on the projective line splits as
a direct sum of line bundles. Nevertheless, vector bundles are still mysterious for higher-dimensional projective spaces. Since projective spaces are covered by lines, it is natural to consider the restriction of vector bundles to lines. This motivates the study of \emph{uniform bundles} on projective spaces, whose splitting type is independent of the chosen line. There are numerous classic and recent papers devoted to this topic \cite{Bal, EHS, Ele78, Ele, ellia, sato, ven}. In fact, the concept of uniform bundles can be generalized to the projective varieties
swept by lines, such as generalized Grassmannians \cite{DFG1, DFG2, FLL, FLL2} and some special Fano manifolds \cite{MOS, MOS2}. On projective spaces, there is also another important class of vector bundles, called \emph{homogeneous bundles}. It is obvious that homogeneous bundles are uniform. Conversely, one may ask whether every uniform vector bundle is homogeneous. The answer is negative.\\
	
	It is known from \cite{Bal, Ele78, ellia} that uniform $(n+1)$-bundles on $\mathbb{P}^n$ either split, or are of the form $\operatorname{Sym}^2T_{\mathbb{P}^2}(a)$ or $T_{\mathbb{P}^n}(b)\oplus \oh_{\mathbb{P}^n}(c)$ (by dualizing) for some $a,b,c\in\mathbb{Z}$ and so they are homogeneous. However, in 1979, Elencwajg \cite{Ele} constructed a non-homogeneous uniform $4$-bundle over $\mathbb{P}^2$. Subsequently, Hirschowitz \cite{OSS} provided examples of uniform non-homogeneous bundles of rank $3n-1$ over $\mathbb{P}^n$ for $n\geq 3$. In 1980, Dr{\'e}zet \cite{Dre} proved the existence of uniform non-homogeneous bundles of rank $2n$ over $\mathbb{P}^n$ for $n\geq 3$. These results on projective spaces naturally lead us to the analogous problem on generalized Grassmannians. To this end, we pose the following problem:
	
	\iffalse
	In 1983, Ballico and Ellia (\cite{EP-MP}) demonstrated that the conjecture holds true for $n = 3$. In 2025, R. Du and the second author proved that the conjecture also holds for $n=4$. Their results imply that for $n\le 4$, $2n-1$ is the largest integer for which all uniform bundles of rank at most $2n-1$ on $\mathbb{P}^n$ are homogeneous. These results on projective spaces naturally lead us to the analogous problem on generalized Grassmannians. To this end, we pose the following problem:
	\fi
	
	\begin{prob}\label{prob}
		Determine the largest integer $k=k(X)$%, depending only on the generalized Grassmannian $X$, 
		such that  uniform bundles on a generalized Grassmannian $X$ of rank at most $k$ are homogeneous.
	\end{prob}
	
	Motivated by Problem \ref{prob}, in this paper we study uniform vector bundles on $n$-dimensional smooth quadric hypersurfaces $\mathbb{Q}^n$, aiming to determine the threshold $k(\mathbb{Q}^n)$. In 1983, Fritzsche \cite{fri} proved that uniform $2$-bundles on $\mathbb{Q}^{3}$ either split, or are precisely the bundles $\mathcal{S}(a)$, where $\mathcal{S}$ is the spinor bundle on $\mathbb{Q}^{3}$ and $a\in \mathbb{Z}$. Subsequently, Guyot \cite{guy} obtained an analogous classification for uniform $2$-bundles on $\mathbb{Q}^{4}$. For $n\geq 5$, Kachi and Sato (see \cite[Theorem 4.1]{KS}) proved that every uniform vector bundle on $\mathbb{Q}^n$ of rank at most $n-2$ (resp. $n-3$) for odd (resp. even) $n$ splits. An alternative proof, relying on a general splitting criterion for low-rank uniform bundles on varieties covered by lines, was given by Mu\~{n}oz, Occhetta, and Sol\'{a} Conde (see \cite[Corollary 3.3]{MOS}).\\
	
	%In 2012, Mu\~{n}oz, Occhetta and Sol\'{a} Conde provided an alternative proof via a general splitting criterion for low-rank uniform bundles  on varieties covered by lines (see \cite[Corollary 3.3]{munoz2012uniform}). \\
	
	In 2025, the first author, together with D. Li and Y. Li \cite{FLL}, extended these results through a systematic analysis of uniform bundles on arbitrary generalized Grassmannians. Specifically, they showed that every unsplit uniform $4$-bundle on $\mathbb{Q}^5$ and $\mathbb{Q}^6$ is isomorphic to a twist of a spinor bundle or its dual. Moreover, they showed that for $n\ge 7$, uniform bundles of rank $n-1$ (resp. $n-2$) over $\mathbb{Q}^n$ split when $n$ is odd (resp. even) (see \cite{FLL2} Theorem 3.11). Since the spinor bundles on $\mathbb{Q}^n$ are homogeneous, the above results imply that for every integer $n\ge 3$, all uniform bundles on $\mathbb{Q}^n$ of rank at most $n-1$ for odd $n$ and $n-2$ for even $n$ are homogeneous. It is natural to ask whether every uniform $n$-bundle on $\mathbb{Q}^n$ is homogeneous, which thereby motivates the classification of uniform $n$-bundles.\\
	
	%This naturally leads to the question whether every uniform $n$-bundle on $\mathbb{Q}^n$ is homogeneous, which forms the central open problem treated in the present paper.\\
	
	The main results of our paper are broadly divided into three parts. The first main result is the classification of uniform  bundles on $\mathbb{Q}^{n}$ of rank $n$ for $n=3$, $5$. 
	
	\begin{theorem}\label{main1}
	Let $n=3$ or $5$. Then every uniform bundle on $\mathbb{Q}^{n}$ of rank $n$  (up to dual) either splits, or is isomorphic to
		\[
		\mathcal{S}(a)\oplus \mathcal{O}_{\mathbb{Q}^{n}}(b),~ T_{\mathbb{Q}^{n}}(c)~(a,b,c\in  \mathbb{Z})\]
		or a twist of the kernel of a bundle epimorphism \[\mathcal{L}^{\perp}\longrightarrow \mathcal{O}_{\mathbb{Q}^{n}},\] where $\mathcal{S}$ is the spinor bundle and
		$\mathcal{L}^{\perp}$ is the orthogonal complement of $\mathcal{O}_{\mathbb{Q}^{n}}(-1)$ with respect to the symmetric bilinear form defining $\mathbb{Q}^n$. 
		
		%$\mathcal{L}_{ker}$ is given by an exact sequence\[0\longrightarrow \mathcal{L}_{ker}\longrightarrow \mathcal{L}^{\perp}\longrightarrow \mathcal{O}_{\mathbb{Q}^{n}}\longrightarrow 0,\]
		
	\end{theorem}
	
	For any quadric $\mathbb{Q}^{n}$, let $\mathcal{L}^{\perp}$ be the orthogonal complement of $\mathcal{O}_{\mathbb{Q}^{n}}(-1)$ introduced above. For each bundle epimorphism
	 $\mathcal{L}^{\perp}\longrightarrow \mathcal{O}_{\mathbb{Q}^{n}}$, its kernel will be called an \emph{orthogonal kernel bundle}, %all of which are uniformly 
	denoted by $\mathcal{L}_{\mathrm{ker}}$. In this paper we show that $\mathcal{L}_{\mathrm{ker}}$ are non-homogeneous. Thus, for $n\ge 4$, the bundles $\mathcal{L}_{\mathrm{ker}}$ on $\mathbb{Q}^{2n-1}$ provide the \textbf{first known examples} of unsplit uniform bundles of minimal rank that fail to be homogeneous among all generalized Grassmannians.
	
	Next, we turn our attention to Problem \ref{prob}. By the classical classification of low-rank uniform bundles on $\mathbb{P}^n$, we know $k(\mathbb{P}^n)\ge n+1$. 
	For a generalized Grassmannian $X$ not isomorphic to $\mathbb{P}^n$, we establish an upper bound for $k(X)$. Combining this with the classification of uniform bundles on $\mathbb{Q}^{2n-1}$ and $\mathbb{Q}^{2n}$ of rank less than $2n-1$, we obtain the following theorem.

	\begin{theorem}\label{main2}
Let $X$ be a generalized Grassmannian. Suppose that $X$ is not isomorphic to a projective space. Then 
\[
k(X)\le \dim X-1.
\]
In particular, $k(\mathbb{Q}^{2n-1}) = 2n-2$ and $2n-2\le k(\mathbb{Q}^{2n})\le 2n-1$.
\end{theorem}
	
\iffalse
	\begin{theorem}[Theorem \ref{non-homo}]
		Every orthogonal kernel bundle $\mathcal{L}_{\mathrm{ker}}$ on $\mathbb{Q}^{n}~(n\ge 3)$ is a uniform non-homogeneous bundle of rank $n$. 
	\end{theorem}
	
	Together with the classification of uniform bundles on $\mathbb{Q}^{2n-1}$ and $\mathbb{Q}^{2n}$ of rank less than $2n-1$, Theorem \ref{main2} determines $k(\mathbb{Q}^{2n-1}) = 2n-2$ exactly and provides the bound $2n-2\le k(\mathbb{Q}^{2n})\le 2n-1$. Notably, this yields 
a complete solution to Problem \ref{prob} for odd-dimensional quadrics.\\
	\fi
	%By the results of ???,all uniform vector bundles of rank less than $2n-1$ on $\mathbb{Q}^{2n-1}$ are homogeneous. Therefore, combining with Theorem \ref{thm1}, we provide a complete solution to Problem \ref{prob} for odd-dimensional quadrics, namely $k(\mathbb{Q}^{2n-1})=2n-2$.
	
	Finally, we give a new characterization of projective spaces from the perspective of uniform vector bundles.
	\begin{theorem}%[Theorem \ref{projective}]
	Let $X$ be a generalized Grassmannian of dimension $n\ge 2$. Then $X\cong\mathbb{P}^n$ if and only if $T_X$ is the unique unsplit uniform bundle on $X$ of minimal rank, up to dual and twist.
	\end{theorem} 
	\section{Preliminaries}\label{pre}
	\subsection{Rational homogeneous spaces and uniform bundles}\label{sec1}
	
	We first review some well-known facts on rational homogeneous spaces and refer to \cite{DFG1, LM, ott} for more details.
	
Let $G$ be a connected simple complex Lie group, $B$ a Borel subgroup of $G$ and $T$ a maximal torus. Denote by $\Phi$ the root system of $G$ determined by $(G,T)$ with a base $\Delta=\{\alpha_1,\ldots,\alpha_n\}$ of simple roots and the subset $\Phi^+$ of positive roots. A closed subgroup $P\subset G$ is called parabolic if it contains some Borel subgroup of $G$. There is a one-to-one correspondence between the subsets $I\subset \{1,2,\ldots,n\}$ and parabolic subgroups $P_I$. The quotient $G/P_I$ is called a \emph{rational homogeneous space}. In particular, when $I=\{k\}$, $P_k$ is a maximal parabolic subgroup of $G$ and $G/P_k$ is called a \emph{generalized Grassmannian}, also called a rational homogeneous space of Picard number $1$. For example, $A_n/P_k$ is the usual Grassmannian $Gr(k,n+1)$.

Given a generalized Grassmannian $X$, we will consider families of lines $\mathcal{M}$ on $X$. Denote by $\mathcal{U}$ the universal family, which has a natural $\mathbb{P}^{1}$-bundle structure over $\mathcal{M}$, i.e. we have the following natural diagram
\begin{align}\label{diagram}
	\xymatrix{
		\mathcal{U}\ar[d]^{p}   \ar[r]^-{q} & \mathcal{M}\\
		X.
	}
\end{align}
Landsberg and Manivel describe $X$, $\mathcal{U}$ and $\mathcal{M}$ by using the marked Dynkin diagram. To be more specific, assume $X$ is $G/P_k$ and let $N(k)$ be the nodes in $\mathcal{D}(G)$ connecting the $k$-th node $\alpha_k$ in $\mathcal{D}(G)$. When $\alpha_k$ is a
long root, $\mathcal{U}$ is $G/P_{k\cup N(k)}$ and $\mathcal{M}$ is $G/P_{N(k)}$. When $\alpha_k$ is a short root, $\mathcal{M}$ is the union of two orbits, an open
orbit and its boundary $G/P_{N(k)}$ (see \cite[Theorem 4.3]{LM} ).\\

Given a vector bundle $E$ on $X$, we say that $E$ \emph{splits} if it decomposes as a direct sum of line bundles, and $E$ is \emph{unsplit} otherwise. In this paper, we are interested in bundles whose splitting type on lines is constant.
 
	\begin{definition}\label{defi}
		A vector bundle $E$ of rank $r$ on $X=G/P_k$ is called uniform if the restriction of $E$ to every line $L$ is isomorphic to $\oh_L(a_1)\oplus\cdots\oplus\oh_L(a_r)$ with  $a_1\geq \cdots \geq a_r$ and $(a_1,\dots,a_r)$ is independent of $L$. The $r$-tuple $(a_1,\dots,a_r)$ is called the splitting type of $E$. 
		%A vector bundle $E$ on $G/P_k$ is called uniform if the restriction of $E$ to every line on $G/P_k$ splits as a direct sum of line bundles with the same splitting type. Moreover, $E$ is said to be uniform with respect to the special family of lines if the same holds for restrictions to every line given by $\mathcal{M}$.
	\end{definition}
	
	%\begin{remark}When $\alpha_k$ is a long root, uniform with respect to the special family of lines is equivalent to uniform in the usual sense. By contrast, if $\alpha_k$ is a short root, the two notions differ, we refer the reader to \cite{} for further details.\end{remark}
	
	\subsection{Homogeneous bundles}
Now we introduce an important class of vector bundles on a rational homogeneous space $G/P$.

\begin{definition}
	Over $G/P$, a vector bundle $E$ is called \emph{homogeneous} if there exists an action of $G$ over $E$ such that the following diagram commutes
	
	\centerline{
		\xymatrix{   G\times E \ar[r]\ar[d]& E\ar[d]\\
			G\times G/P \ar[r] & G/P. }}
\end{definition}
In fact, the category of $P$-modules is equivalent to that of homogeneous bundles on $G/P$. Given an irreducible $P$-module $V$, the associated bundle $E=G\times_{P}V$ is called an \emph{irreducible homogeneous vector bundle}. 
Let $\Delta=\{\alpha_1,\ldots,\alpha_n\}$ be the set of simple roots and $\lambda_{1},\ldots,\lambda_{n}$ the corresponding fundamental weights. Let $P$ be the parabolic subgroup of $G$ defined by $I\subset \{1,2,\ldots,n\}$. Then all irreducible $P$-modules can be classified as follows (cf. \cite[Proposition 10.9]{ott}):
\[V\otimes \bigotimes_{i\in I}L^{n_i}_{\lambda_i},\]
where $V$ is a representation of $S_P$ (the semisimple part of $P$), $n_i\in\mathbb{Z}$ and $L_{\lambda_i}$ is the one-dimensional representation of $P$ with weight $\lambda_i$.
Notice that the weight lattice of $S_P$ can be embedded in the weight lattice of $G$. If $\lambda'$ is the highest weight of $V$ as a $S_P$-module, we will say that $\lambda'+\sum\limits_{i\in I}{n_i}{\lambda_i}$ is the highest weight of the $P$-module $V\otimes \bigotimes_{i\in I}L^{n_i}_{\lambda_i}$. \\

In this paper, we denote by $E_\lambda$ the irreducible homogeneous bundle corresponding to the irreducible $P$-module with highest weight $\lambda$. Write $\lambda=\sum \limits_{i=1}^{n} a_i\lambda_i$. Then $a_i\geq0$ for $i\notin I$, as the highest weight of any irreducible $S_P$-module is dominant.
The powerful tool to compute the cohomology of irreducible homogeneous bundles over $G/P$ is the following theorem. First, some definitions are presented.
	%%%%%%%%%%%%%%%%%%
	\begin{definition}\ Let $\lambda$ be a weight.\\
		(i)\ $\lambda$ is called \emph{singular} if there exists $\alpha \in \Phi^{+}$ such that $(\lambda,\alpha)=0$.\\
		(ii)\ $\lambda$ is called \emph{regular of index} $p$ if it is not singular and if there are exactly $p$ roots $\alpha_{1},\ldots,\alpha_{p} \in \Phi^{+}$ such that $(\lambda,\alpha_{i})<0$.
	\end{definition}
	
	%Now we can introduce the Borel--Bott--Weil Theorem.
	
	\begin{theorem}[Borel--Bott--Weil, see \cite{ott} Theorem 11.4]\label{borel bott weil} Let $E_\lambda$ be an irreducible homogeneous vector bundle over $G/P.$
		\begin{enumerate}
			\item[1)] If $\lambda+\rho$ is singular, then $$H^i(G/P,E_\lambda)=0, \forall i\in\mathbb{Z}.$$
			\item[2)] If $\lambda+\rho$ is regular of index $p$, then $$H^i(G/P,E_\lambda)=0, \forall i\neq p,$$
			and $$H^p(G/P,E_\lambda)=G_{w(\lambda+\rho)-\rho},$$ where $\rho=\sum\limits_{i=1}^n\lambda_i$, $w(\lambda+\rho)$ is the unique element of the fundamental Weyl chamber of $G$ which is congruent to $\lambda+\rho$ under the action of the Weyl group and $G_{w(\lambda+\rho)-\rho}$ is the irreducible representation of $G$ with highest weight $w(\lambda+\rho)-\rho$.
		\end{enumerate}
	\end{theorem}
	
	\subsection{Vector bundles on quadrics}\label{sec3}
In this section, we shall concentrate on several important vector bundles on the $n$-dimensional smooth quadric $\mathbb{Q}^n$, which is defined by a non-degenerate symmetric bilinear form $\mathcal{Q}$ on $\mathbb{C}^{n+2}$. Recall that $\mathbb{Q}^{2n-1}=B_n/P_1$ and $\mathbb{Q}^{2n}=D_{n+1}/P_1$.

The tangent bundle on $\mathbb{Q}^n$
is known to be homogeneous and 
\[
T_{\mathbb{Q}^n}=\begin{cases}
	E_{\lambda_2},&\text{if}~X=\mathbb{Q}^n~(n\ge 5),\\
	E_{\lambda_1+\lambda_3},&\text{if}~X=\mathbb{Q}^4,\\
	E_{2\lambda_2},&\text{if}~X=\mathbb{Q}^3.
\end{cases}
\]
Note that $\mathbb{Q}^4 \cong G(2,4)=A_3/P_2$ and the simple Lie algebra of type $D_3$ is exactly the simple Lie algebra of type $A_3$.

%There is another important class of vector bundles called
Another important family consists of the
 \emph{spinor bundles}, constructed from the spin and half-spin representations of the orthogonal Lie algebra. Clearly, they are homogeneous. Specifically, $E_{\lambda_n}$ is the unique spinor bundle on $\mathbb{Q}^{2n-1}$. On $\mathbb{Q}^{2n}$, there are two nonisomorphic spinor bundles $E_{\lambda_{n}}$ and $E_{\lambda_{n+1}}$. We refer to \cite{ott1, ott2} for more details on spinor bundles.

We now introduce the core vector bundle of this paper on $\mathbb{Q}^n$, namely the \emph{$\mathcal{Q}$-orthogonal complement} of $\mathcal{O}_{\mathbb{Q}^{n}}(-1)$, denoted by $\mathcal{L}_n^{\perp}$. We omit the subscript $n$ when there is no ambiguity. As is well known, $T_{\mathbb{Q}^n}(-1)=\mathcal{L}^{\perp}/\mathcal{O}_{\mathbb{Q}^{n}}(-1)$, so $\mathcal{L}^{\perp}$ fits into the exact sequence
\begin{align}\label{ex1}
0\rightarrow  \mathcal{O}_{\mathbb{Q}^{n}}(-1)\rightarrow \mathcal{L}^{\perp}\rightarrow T_{\mathbb{Q}^n}(-1)\rightarrow 0.
\end{align}
We now present several fundamental properties of the bundle $\mathcal{L}^{\perp}$.
\begin{proposition}\label{prop1}
Let $\mathcal{L}^{\perp}$ be the $\mathcal{Q}$-orthogonal complement of $\mathcal{O}_{\mathbb{Q}^{n}}(-1)$ on $\mathbb{Q}^n~(n\ge 3)$. Then
\begin{enumerate}
	\item %The Chern classes of $\mathcal{L}^{\perp}$ satisfy  $c_{n+1}(\mathcal{L}^{\perp})=0$
	$c_i(\mathcal{L}^{\perp})=(-1)^iX_1^i$ for $1\le i\le n$, where $X_1$ is the ample generator of $\mathrm{Pic}(\mathbb{Q}^{n})$;
	\item $\mathcal{L}^{\perp}$ is isomorphic to $\Omega_{\mathbb{P}^{n+1}}(1)|_{\mathbb{Q}^{n}}$;
	\item $\mathcal{L}^{\perp}$ is stable and $(\mathcal{L}^{\perp})^{\vee} $ is generated by global sections;
	\item $\mathcal{L}^{\perp}$ is uniform of splitting type $(0,\ldots,0,-1)$;
	\item All cohomology groups of $\mathcal{L}^{\perp}$ vanish, i.e., $H^i(\mathbb{Q}^{n},\mathcal{L}^{\perp})=0$ for all $i$;
	\item $H^0(\mathbb{Q}^{n},(\mathcal{L}^{\perp})^{\vee})=n+2$ and $H^i(\mathbb{Q}^{n},(\mathcal{L}^{\perp})^{\vee})=0$ for any $i>0$.
\end{enumerate}
\end{proposition}
 \begin{proof}
 	(1) Note that $c_i(T_{\mathbb{Q}^{n}}(-1))=0$ if $i$ is odd and $c_i(T_{\mathbb{Q}^{n}}(-1))=X_1^i$ if $i$ is even. Then the Chern classes of $\mathcal{L}^{\perp}$ follow at once from the exact sequence (\ref{ex1}) and Whitney's formula.
 	
 (2) Since $\mathbb{Q}^{n}\subset \mathbb{P}^{n+1}$ has a normal bundle $N_{\mathbb{Q}^{n}/\mathbb{P}^{n+1}}=\oh_{\mathbb{Q}^{n}}(2)$, one has the exact sequence $0\rightarrow T_{\mathbb{Q}^{n}} \rightarrow T_{\mathbb{P}^{n+1}}|_{\mathbb{Q}^{n}} \rightarrow \oh_{\mathbb{Q}^{n}}(2) \rightarrow 0$. Tensoring this exact sequence with $\oh_{\mathbb{Q}^{n}}(-1)$ and dualizing, we obtain $0\rightarrow \oh_{\mathbb{Q}^{n}}(-1) \rightarrow \Omega_{\mathbb{P}^{n+1}}(1)|_{\mathbb{Q}^{n}}\rightarrow T_{\mathbb{Q}^{n}}(-1) \rightarrow 0$. (Here we use the isomorphism $T_{\mathbb{Q}^{n}}(-1)\cong \Omega_{\mathbb{Q}^{n}}(1)$.) 
 It follows that $\mathcal{L}^{\perp}$ and $\Omega_{\mathbb{P}^{n+1}}(1)|_{\mathbb{Q}^{n}}$ are non-trivial extensions of 
 $T_{\mathbb{Q}^{n}}(-1)$ by $\oh_{\mathbb{Q}^{n}}(-1)$. Nevertheless, by the Borel-Bott-Weil theorem, $\dim \operatorname{Ext}^1(T_{\mathbb{Q}^{n}}(-1),\oh_{\mathbb{Q}^n}(-1))=h^1(T_{\mathbb{Q}^{n}}(-2))=1$. So $\mathcal{L}^{\perp}\cong\Omega_{\mathbb{P}^{n+1}}(1)|_{\mathbb{Q}^{n}}$.
 
 (3) As $T_{\mathbb{P}^{n+1}}(-1)$ is globally generated, so is its restriction to $\mathbb{Q}^{n}$. In view of (2), $(\mathcal{L}^{\perp})^{\vee}\cong T_{\mathbb{P}^{n+1}}(-1)|_{\mathbb{Q}^{n}}$ is consequently globally generated. Furthermore, by \cite[Theorem A and B]{BCM} and \cite[Theorem 1.7]{Liu}, $T_{\mathbb{P}^{n+1}}|_{\mathbb{Q}^{n}}$ is stable, which implies the stability of $(\mathcal{L}^{\perp})^{\vee}$, and hence $\mathcal{L}^{\perp}$ is stable as well.
 
 (4) By the exact sequence (\ref{ex1}), we have $c_1(\mathcal{L}^{\perp})|_{\ell}=-1$ for every line $\ell\subset \mathbb{Q}^{n}$. Combined with (3), this gives the result immediately.
 
 (5) From the long exact sequence in cohomology induced by (\ref{ex1}), we immediately obtain $H^i(\mathbb{Q}^{n},\mathcal{L}^{\perp})\cong H^i(\mathbb{Q}^{n},T_{\mathbb{Q}^{n}}(-1))$ for every $i$. By the Borel-Bott-Weil theorem, the cohomology groups $H^i(\mathbb{Q}^{n},T_{\mathbb{Q}^{n}}(-1))$ vanish for all $i$, whence the assertion (5) holds. Furthermore, by dualizing the exact sequence (\ref{ex1}) and passing to the long exact sequence, we deduce (6).
 \end{proof}
	
	\section{Uniform bundles on $\mathbb{Q}^3$ and $\mathbb{Q}^5$}
	
	\iffalse
	 Let $\mathcal{S}$ is the spinor bundle on $\mathbb{Q}^{2n-1}$, which is precisely the homogeneous $E_{\lambda_n}$. In 1983, Fritzsche (\cite{EP-MP}) proved that uniform $2$-bundles on $\mathbb{Q}^{3}$ either split, or are the bundles $\mathcal{S}(a)~(a\in\mathbb{Z})$. Subsequently, Kachi and Sato (\cite{}) showed that all uniform bundles of rank less than $4$ on $\mathbb{Q}^{5}$ split, whereas Fang-Li-Li (\cite{}) proved that those of rank $4$ either split or are the bundles $\mathcal{S}(a)$. 
	 \fi
	 As stated in the introduction, uniform bundles on $\mathbb{Q}^{2n-1}$ of rank at most $2n-2$ have been fully classified. However, the classification of rank $2n-1$ uniform bundles remains open. In this section, we classify uniform bundles of rank $3$ on $\mathbb{Q}^{3}$ and rank $5$ on $\mathbb{Q}^{5}$. 
	 
	\subsection{Uniform $3$-bundles on $\mathbb{Q}^3$}
	We begin with a useful lemma concerning uniform vector bundles on an arbitrary generalized Grassmannian.
	
	\begin{lemma}\label{lemma1}
	Let $X$ be a generalized Grassmannian  %with long root $\alpha_k$ 
	and $E$ a uniform bundle on $X$ of splitting type $(a_1,\dots,a_r)$. If
	 %$(\underbrace{a_1,\dots,a_1}_{l_1},\dots, \underbrace{a_k,\dots,a_k}_{l_k}), ~a_1>\cdots> a_k$. If
	 %$E|L\cong \mathcal{O}_{L}(a_1)^{r_1}\oplus\cdots\oplus \mathcal{O}_{L}(a_k)^{r_k}, ~a_1>\cdots> a_k$. If 
	\[
	a_s-a_{s+1}\geq 2~\text{ for some } ~s<r,
	\]
then $E$ is an extension of two uniform bundles of splitting types 
$(a_1,\dots,a_s)$ and $(a_{s+1},\dots,a_r)$, respectively.
 %$(\underbrace{a_1,\dots,a_1}_{l_1},\dots, \underbrace{a_s,\dots,a_s}_{l_s})$ and $(\underbrace{a_{s+1},\dots,a_{s+1}}_{l_{s+1}},\dots, \underbrace{a_k,\dots,a_k}_{l_k})$, respectively.
	\end{lemma}
	\begin{proof}
Let $L$ be a line in $X$. We denote the inclusion morphism by $f_L\colon L(\cong \mathbb{P}^1)\rightarrow X$. As $X$ is a homogeneous variety, the tangent bundle $T_X$ is globally generated. Then $f_L^*(T_X)$ is also globally generated and hence $H^1(L,f_L^*(T_X))$ vanishes. So Mor$(\mathbb{P}^1,X)$ is smooth at $[f_L]$.

If $a_s-a_{s+1}\geq 2$ for some $s<r$, then by the proof of \cite[Proposition 3.1]{PRT}, there exists a subbundle $W$ of $E$ such that $W$ is uniform of splitting type $(a_1,\dots,a_s)$ and $U:=E/W$ is a uniform bundle of splitting type 
$(a_{s+1},\dots,a_r)$. This completes the proof of the lemma.
	\end{proof}
	
	\begin{theorem}\label{uniform Q3}
		Let $E$ be a uniform bundle of rank $3$ on $\mathbb{Q}^{3}$. Then up to dual, $E$ either splits, or is isomorphic to
		\[
		\mathcal{S}(a)\oplus \mathcal{O}_{\mathbb{Q}^{3}}(b),~ T_{\mathbb{Q}^{3}}(c)~(a,b,c\in  \mathbb{Z})\]
		or a twist of the kernel of a bundle epimorphism $\mathcal{L}^{\perp}\longrightarrow \mathcal{O}_{\mathbb{Q}^{3}}$, where
		$\mathcal{L}^{\perp}$ is the $\mathcal{Q}$-orthogonal complement of $\mathcal{O}_{\mathbb{Q}^{3}}(-1)$.
	\end{theorem}
	\begin{proof}
Let $E$ be an unsplit uniform bundle of splitting type $(a_1,a_2,a_3)$ on $\mathbb{Q}^{3}$. By Lemma \ref{lemma1}, if there is an integer $s<3$ such that $a_s-a_{s+1}\geq 2$, then $E$ is an extension of uniform bundles of rank $1$ and $2$. %Since uniform $2$-bundles on $\mathbb{Q}^{3}$ either split or are the bundles $\mathcal{S}(a)~(a\in\mathbb{Z})$
According to the classification of uniform $2$-bundles on $\mathbb{Q}^{3}$, $E$ either splits or is isomorphic to $\mathcal{S}(a)\oplus \mathcal{O}_{\mathbb{Q}^{3}}(b)~(a,b\in \mathbb{Z})$, as both line bundles and the spinor bundle $\mathcal{S}$ are ACM bundles (see \cite[Corollary 3.12]{DFR} or \cite[Theorem 3.5]{ott2}). Hence up to dual and twisting $E$ by a line bundle, we may assume the splitting type of $E$ is one of the following types
$(0,0,0), ~(0,0,-1), ~(1,0,-1)$. We can rule out the possibility $(0,0,0)$ by \cite[Proposition 1.2]{AW}.
%$a_i-a_{i+1}\le 1$ for any $1\le i\le 2$.

 We adopt the notations of Section \ref{sec1}. In the case $X=\mathbb{Q}^{3}(=B_2/P_1)$, $\mathcal{M}=B_2/P_2$ is the family of all lines on $\mathbb{Q}^{3}$ and $\mathcal{U}$ is $B_2/P_{1,2}$. Moreover, every $q$-fiber and every $p$-fiber is isomorphic to $\mathbb{P}^1$. %In particular, $\mathbb{Q}^{3}$ is the family of special lines on $\mathbb{P}^{3}$. 
 By \cite[Lemma 2.1]{FLL}, we list the cohomology rings as follows
\begin{equation*}
	\begin{aligned}
		&H^\bullet(\mathbb{Q}^{3},\mathbb{Q})\cong\mathbb{Q}[X_1]/(X_1^4),~H^{\bullet}(\mathcal{U},\mathbb{Q})\cong\mathbb{Q}[X_1,X_2]/I,~ I=\bigl(\Sigma_d(X_1^2,X_2^2)_{1\leq d\leq 2}\bigr),\\
		&H^{\bullet}(\mathcal{M},\mathbb{Q})\cong\mathbb{Q}[X_1+X_2,X_1X_2]/I\cong \mathbb{Q}[X_1+X_2]/\bigl((X_1+X_2)^4\bigr).
	\end{aligned}
\end{equation*}
Note that $\mathcal{M}\cong \mathbb{P}^3$, we have $c_1(q^*\mathcal{O}_{\mathcal{M}}(1))=\frac{1}{2}(X_1+X_2)$. And we denote the Chern classes $c_t(p^*E)(1\leq t \leq 3)$ of $p^*E$ by $\mu_tX_1^t\pmod{I}$ for some rational numbers $\mu_t$.\\

\textbf{Case I}: the splitting type of $E$ is $(0,0,-1)$.\\

The relative Harder-Narasimhan (H-N) filtration induces an exact sequence
\begin{align}\label{exact1}
	0\rightarrow q^*G \rightarrow p^*E \rightarrow q^*L\otimes p^*\oh_X(-1) \rightarrow 0,
\end{align} where $G$ is of rank $2$  and $L$ is a line bundle. In view of $c_1(p^*E)=c_1(q^*G)+c_1(q^*L\otimes p^*\oh_X(-1))$, we set $c_1(q^*G)=a(X_1+X_2)$, $c_1(q^*L\otimes p^*\oh_X(-1))=-a(X_1+X_2)-X_1$ with $2a\in \mathbb{Z}$. We further write $c_2(q^*G)=b(X_1+X_2)^2$.
%Note that $X_1^2+X_2^2=0$ on $\mathcal{U}$. 
%Comparing the coefficient of $X_1X_2$ on both sides of the equation 
By Whitney's formula, we have 
\begin{align*}
\mu_2X_1^2&=c_2(p^*E)=c_2(q^*G)+c_1(q^*G)c_1(q^*L\otimes p^*\oh_X(-1))\\
&=-aX_1^2+(2b-2a^2-a)X_1X_2,
\end{align*}
where the final equality uses $X_1^2+X_2^2=0$ in $H^{\bullet}(\mathcal{U},\mathbb{Q})$. Matching the coefficient of $X_1X_2$ then gives $2b-2a^2-a=0$. Similarly,
%$\mu_2X_1^2=c_2(p^*E)=c_2(q^*G)+c_1(q^*G)c_1(q^*L\otimes p^*\oh_X(-1))=-aX_1^2+(2b-2a^2-a)X_1X_2$ (noting that here we using $X_1^2+X_2^2=0$ on $H^{\bullet}(\mathcal{U},\mathbb{Q})$), 
$\mu_3X_1^3=c_3(p^*E)=c_2(q^*G)c_1(q^*L\otimes p^*\oh_X(-1))=2abX_1^3-(2ab+2b)X_1^2X_2$ yields $2ab+2b=0$ after equating 
the coefficient of $X_1^2X_2$. Solving these equations, we obtain
$a=0$, $-\frac{1}{2}$ or $-1$. Accordingly,
$L=\oh_{\mathcal{M}}$, $\oh_{\mathcal{M}}(1)$ or $\oh_{\mathcal{M}}(2)$. \\

If $a=0$, then $c_1(q^*G|_{p^{-1}(x)})$ and $c_1(q^*L\otimes p^*\oh_X(-1)|_{p^{-1}(x)})$ vanish for $x\in X$. By \cite[Lemma 3.1]{FLL}, $q^*G|_{p^{-1}(x)}$ and $q^*L\otimes p^*\oh_X(-1)|_{p^{-1}(x)}$ are trivial. By the base change theorem, $p_{*}q^*G$ is a rank $2$ bundle and $p_{*}(q^*L\otimes p^*\oh_X(-1))$ is a line bundle. They are uniform bundles by construction, of splitting types $(0,0)$ and $(-1)$ respectively, so they split. Therefore, $E$ is an extension of two split bundles, which implies that $E$ splits.\\

If $a=-\frac{1}{2}$, then $L=\oh_{\mathcal{M}}(1)$. Hence $q^*L\otimes p^*\oh_X(-1)|_{p^{-1}(x)}\cong \oh_{\mathbb{P}^1}(1)$ and $q^*G|_{p^{-1}(x)}\cong\oh_{\mathbb{P}^1}\oplus\oh_{\mathbb{P}^1}(-1)$. Therefore $R^1p_*q^*G$ vanishes by the base change theorem. By projecting (\ref{exact1}) to $X$, we obtain an exact sequence $0\rightarrow p_*q^*G \rightarrow E \rightarrow p_*q^*\oh_{\mathcal{M}}(1)\otimes \oh_X(-1)  \rightarrow 0$. According to \cite[Proposition 10.13]{ott}, the sheaf $p_*q^*\oh_{\mathcal{M}}(1)$ is the homogeneous bundle $E_{\lambda_2}$, i.e., the spinor bundle 
$\mathcal{S}$. As $\mathcal{S}$ is ACM, we have $E\cong\mathcal{S}(-1)\oplus \mathcal{O}$, with Chern classes $c_1(E)=-X_1$, $c_2(E)=\frac{1}{2}X_1^2$, $c_3(E)=0$.\\
 
If $a=-1$, then $b=\frac{1}{2}$. In this case, $G$ is a $2$-bundle on $\mathbb{P}^3$ with $c_1(G)=-2$, $c_2(G)=2$ and
 \begin{align*}%\label{equ1}
 c_1(E)=-X_1, ~c_2(E)=X_1^2 ~\text{and}~c_3(E)=-X_1^3.
 \end{align*}
In addition, $L=\oh_{\mathcal{M}}(2)$. Hence $q^*L\otimes p^*\oh_X(-1)|_{p^{-1}(x)}\cong \oh_{\mathbb{P}^1}(2)$ and $q^*G|_{p^{-1}(x)}\cong\oh_{\mathbb{P}^1}^{\oplus 2}(-1)$ or $\oh_{\mathbb{P}^1}\oplus \oh_{\mathbb{P}^1}(-2)$.\\

\begin{claim}\label{claim:1}
There exists a point $x$ for which $q^*G|_{p^{-1}(x)}\cong \oh_{\mathbb{P}^1}^{\oplus 2}(-1)$, and a point $y$ for which $\oh_{\mathbb{P}^1}\oplus \oh_{\mathbb{P}^1}(-2)$.
\end{claim}

Proof. Assume for any $x\in X$, $q^*G|_{p^{-1}(x)}\cong \oh_{\mathbb{P}^1}^{\oplus 2}(-1)$, we have $p_*q^*G$ and $R^1p_*q^*G$ vanish by the base change theorem. By projecting (\ref{exact1}) to $X$, we get $E$ is isomorphic to  $p_*q^*\oh_{\mathcal{M}}(2)\otimes \oh_X(-1)$, which is the homogeneous bundle $E_{2\lambda_2}(-1)$, i.e., $T_{\mathbb{Q}^{3}}(-1)$. However, the splitting type of $T_{\mathbb{Q}^{3}}(-1)$ is $(1,0,-1)$, which differs from that of $E$. This immediately yields a contradiction.

Assume for any $x\in X$, $q^*G|_{p^{-1}(x)}\cong \oh_{\mathbb{P}^1}\oplus \oh_{\mathbb{P}^1}(-2)$. Then $p_*q^*G$ and $R^1p_*q^*G$ are both line bundles. By projecting (\ref{exact1}) to $X$, we get an exact sequence
$0\rightarrow p_*q^*G \rightarrow E \rightarrow T_{\mathbb{Q}^{3}}(-1) \rightarrow R^1p_*q^*G\rightarrow0$. Since $c_1(E)=-X_1$ and $c_1(T_{\mathbb Q^3}(-1))=0$, comparing the first Chern classes gives
 $R^1p_*q^*G=\oh_X(a)$, $p_*q^*G=\oh_X(a-1)$ for some integer $a$. Since the splitting type of $E$ is $(0,0,-1)$, we have $a-1\le 0$. On the other hand, the stability of $T_{\mathbb{Q}^{3}}(-1)$ and the fact that its slope is zero yield $a\ge 1$. This forces $a=1$. Since $c_1(T_{\mathbb Q^3}(-1))=0$ and $c_2(T_{\mathbb Q^3}(-1))=X_1^2$, the above exact sequence would give $c_2(E)=2X_1^2$, which contradicts the earlier computation of $c_2(E)$.\\
 
 %Note that the Chern classes of $T_{\mathbb{Q}^{3}}(-1)$ are$c_1(T_{\mathbb{Q}^{3}}(-1))=c_3(T_{\mathbb{Q}^{3}}(-1))=0$ and $c_2(T_{\mathbb{Q}^{3}}(-1))=X_1^2$. From the above exact sequence, one would have 

By the semicontinuity theorem, the set $V_E=\{y\in X\mid q^*G|_{p^{-1}(y)}\cong\oh_{\mathbb{P}^1}\oplus \oh_{\mathbb{P}^1}(-2)\}$ is closed in $X$. Since $p_*q^*G$ is torsion free and zero outside of $V_E$, $p_*q^*G$ vanishes, we have the following short exact sequence:
\begin{align}\label{exact2}
0\rightarrow E \rightarrow T_{\mathbb{Q}^{3}}(-1) \xrightarrow{\varphi} R^1p_*q^*G\rightarrow 0.
\end{align}
On the other hand, by Claim \ref{claim:1} and the semicontinuity theorem, $G$ splits as $\oh_{\mathbb{P}^1}^{\oplus 2}(-1)$ on a generic line of $\mathbb{P}^{3}$. Let $N:=G(1)$. Then $N$ is trivial when restricted to a general line. By \cite[Lemma 2.2.1]{OSS}, $N$ is semistable. Since $c_1(G)=-2$ and $c_2(G)=2$, we have $c_1(N)=0$, $c_2(N)=1$, which implies that $N$ must be stable. (Otherwise, $N$ has a non-zero section $s$, which induces an exact sequence $0\rightarrow \oh_{\mathbb{P}^{3}} \rightarrow N \rightarrow I_Y\rightarrow0$, where the zero locus $Y$ of $s$ is a locally complete curve. Then we have $\omega_Y \cong (\omega_{\mathbb{P}^3} \otimes \det N)|_Y \cong \mathcal{O}_Y(-4)$. However, since $c_2(N)=1$, it follows that $Y$ is a $\mathbb{P}^1$. That is a contradiction to $\omega_Y=\oh_Y(-4)$.) Now we know $N$ is a stable $2$-bundle on $\mathbb{P}^3$ with $c_1(N)=0$, $c_2(N)=1$, hence $N$ is a null correlation bundle and is given by $0\rightarrow N \rightarrow T_{\mathbb{P}^3}(-1) \rightarrow \oh_{\mathbb{P}^3}(1)\rightarrow0$ (see \cite[Lemma 4.3.2]{OSS}). Therefore, $G$ fits into the exact sequence $0\rightarrow G \rightarrow T_{\mathbb{P}^3}(-2) \rightarrow \oh_{\mathbb{P}^3}\rightarrow0$ on $\mathbb{P}^3$. Applying $p_{*}q^{*}$ to this exact sequence yields a new exact sequence (here we use $p_{*}q^{*}G=0$ and $R^1p_{*}q^{*}(T_{\mathbb{P}^3}(-2))=0$)
\[
0\rightarrow p_{*}q^{*}(T_{\mathbb{P}^3}(-2)) \rightarrow \oh_{\mathbb{Q}^3}\rightarrow R^1p_{*}q^{*}G\rightarrow 0.
\]
 From $(\ref{exact2})$, we see $c_1(R^1p_{*}q^{*}G)=c_1(T_{\mathbb{P}^3}(-1))-c_1(E)=X_1$. It follows that $c_1(p_{*}q^{*}(T_{\mathbb{P}^3}(-2)))=-X_1$. %Then by the base-change theorem,
 %By comparing the first Chern classes, one easily deduces that the line bundle 
 So $p_{*}q^{*}(T_{\mathbb{P}^3}(-2))\cong\oh_{\mathbb{Q}^3}(-1)$. This gives a non-trivial extension $\alpha\in \mathrm{Ext}^1(R^1p_{*}q^{*}G,\oh_{\mathbb{Q}^3}(-1))$.\\
 
 Applying the functor $\operatorname{Hom}(-,\oh_{\mathbb{Q}^3}(-1))$ to the exact sequence $(\ref{exact2})$, we obtain a morphism $g\colon\mathrm{Ext}^1(R^1p_{*}q^{*}G,\oh_{\mathbb{Q}^3}(-1))\rightarrow \mathrm{Ext}^1(T_{\mathbb{Q}^{3}}(-1),\oh_{\mathbb{Q}^3}(-1))$. By Proposition \ref{prop1}, the unique non-trivial extension in $\mathrm{Ext}^1(T_{\mathbb{Q}^{3}}(-1),\oh_{\mathbb{Q}^3}(-1))$ gives rise to
 $\mathcal{L}^{\perp}$ up to isomorphism.
 %By the isomorphism $T_{\mathbb{Q}^{3}}(-1)\cong \Omega_{\mathbb{Q}^{3}}(1)$and the Borel-Bott-Weil Theorem, one can show that $\dim Ext^1(T_{\mathbb{Q}^{3}}(-1),\oh_{\mathbb{Q}^3}(-1))=h^1(T_{\mathbb{Q}^{3}}(-2))=1$. Hence the only non-trivial extension is isomorphic to $\mathcal{L}^{\perp}$ (the $\mathcal{Q}$-orthogonal complement of $\mathcal{O}_{\mathbb{Q}^{3}}(-1)$), as $T_{\mathbb{Q}^{3}}(-1)=\mathcal{L}^{\perp}/\oh_{\mathbb{Q}^3}(-1)$. 
 If $g(\alpha)\ne 0$, then we have the following commutative diagram
 \begin{align*}
 \xymatrix{
 	0\ar[r]& \mathcal{O}_{\mathbb{Q}^{3}}(-1)\ar[r]\ar[d]^{\cong}&\mathcal{L}^{\perp}\ar[d]^\phi\ar[r]& T_{\mathbb{Q}^{3}}(-1)\ar[d]^\varphi\ar[r]&0\\
 	0\ar[r]& \mathcal{O}_{\mathbb{Q}^{3}}(-1)\ar[r]&\mathcal{O}_{\mathbb{Q}^{3}}\ar[r] &R^1p_{*}q^{*}G\ar[r]&0,
 }
 \end{align*}
where $\varphi$ denotes the surjection in $(\ref{exact2})$. By the Snake Lemma, $E$ is isomorphic to $\ker \phi$ and hence $E$ fits into the exact sequence
 \begin{align*}%\label{perp}
 0\rightarrow E \rightarrow \mathcal{L}^{\perp}\rightarrow \mathcal{O}_{\mathbb{Q}^{3}}\rightarrow 0.
 \end{align*}

By the same argument, we may rule out the case $g(\alpha)= 0$. Otherwise, the bundle $E$ would satisfy an exact sequence $0\rightarrow E \rightarrow \mathcal{O}_{\mathbb{Q}^{3}}(-1)\oplus T_{\mathbb{Q}^{3}}(-1)\rightarrow \mathcal{O}_{\mathbb{Q}^{3}}\rightarrow 0$, which is incompatible with the splitting type of $E$ on any line.\\

\textbf{Case II}: the splitting type of $E$ is $(1,0,-1)$.\\

The relative H-N filtration of $p^*E$ induces exact sequences
\begin{align}
	&0\rightarrow E_1(=q^*L_1\otimes p^*\oh_{X}(1))\rightarrow  E_2\rightarrow E_2/E_1(=q^*{L_2})\rightarrow 0,\label{exact3-1}\\
	&0\rightarrow E_2\rightarrow  p^*E\rightarrow E_3(=q^*L_3\otimes p^*\oh_{X}(-1))\rightarrow 0, \label{exact3-2}
\end{align} where every $L_i~(i=1,2,3)$ is a line bundle. 
Assume $c_1(E_1)=a_1(X_1+X_2)+X_1$ and $c_1(E_3)=a_2(X_1+X_2)-X_1$. Then $c_1(E_2/E_1)=-(a_1+a_2)(X_1+X_2)$.
By an analysis analogous to that of the Chern polynomials in Case I, we deduce that $(a_1,a_2)=(0,0)$, $(-\frac{1}{2},0)$, $(0,\frac{1}{2})$
or $(-1,1)$.\\

If $(a_1,a_2)=(0,0)$ or $(-\frac{1}{2},0)$, then $L_3=\oh_{\mathcal{M}}$ and for every $x\in X$, $E_2|_{p^{-1}(x)}$ is trivial. By the base change theorem, $p_*{E_2}$ is a rank $2$ bundle and it is uniform of splitting type $(1,0)$. Hence $p_*{E_2}$ is either $\oh_X(1)\oplus \oh_X$ or the spinor bundle $\mathcal{S}$. By projecting the exact sequence (\ref{exact3-2}) to $X$, we obtain an exact sequence $0\rightarrow p_*{E_2}\rightarrow  E\rightarrow \oh_{X}(-1)\rightarrow 0$. Since line bundles and $\mathcal{S}$ are ACM bundles, we have $E$ either splits or is isomorphic to $\mathcal{S}\oplus \oh_X(-1)$.\\

If $(a_1,a_2)=(0,\frac{1}{2})$, then $L_3=\oh_{\mathcal{M}}(1)$ and for every $x\in X$, $E_2|_{p^{-1}(x)}\cong \oh(-1)\oplus \oh$. Therefore $R^1p_*E_2$ vanishes. By projecting (\ref{exact3-2}) to $X$, we obtain an exact sequence $0\rightarrow p_*{E_2} \rightarrow E \rightarrow p_*q^*\oh_{\mathcal{M}}(1)\otimes \oh_X(-1)(=\mathcal{S}(-1))\rightarrow 0$. By comparing the first Chern classes, we have $p_*{E_2}\cong \oh_X(1)$. Hence $E$ is isomorphic to $\mathcal{S}(-1)\oplus \oh_X(1)$.\\

If $(a_1,a_2)=(-1,1)$, then $E_1=q^*\oh_{\mathcal{M}}(-2)\otimes p^*\oh_X(1)$, $E_3=q^*\oh_{\mathcal{M}}(2)\otimes p^*\oh_X(-1)$ and $E_2/E_1=\oh_{\mathcal{U}}$. Since for any $y\in \mathcal{M}$, $E_1|_{q^{-1}(y)}\cong \oh(1)$, we have $R^iq_{*}E_1=0$ for any $i>0$. By the degenerate spectral sequence and the Borel-Bott-Weil theorem, $h^1(\mathcal{U},E_1)=h^1(\mathcal{M},q_{*}E_1)=1$. Thus, the exact sequence (\ref{exact3-1}) either splits or only has a non-trivial extension.
\begin{itemize}
	\item Suppose the exact sequence (\ref{exact3-1}) splits. Then $E_2\cong \oh_{\mathcal{U}}\oplus E_1$. Therefore, $p^*E$ admits $\oh_{\mathcal{U}}$ as a subbundle. %Let $F$ be the quotient bundle $p^*E/\oh_{\mathcal{U}}$. 
	Pushing forward the exact sequence $0\rightarrow \oh_{\mathcal{U}} \rightarrow p^*E\rightarrow F\rightarrow 0$, where $F:=p^*E/\oh_{\mathcal{U}}$,
	to $X$ yields the exact sequence $0\rightarrow \oh_X \rightarrow E\rightarrow p_*F\rightarrow 0$. 
	It is not hard to see that $p_*F$ is a uniform $2$-bundle of splitting type $(1,-1)$. It follows that $p_*F\cong \oh_X(1)\oplus\oh_X(-1)$ and hence $E$ splits, contradicting our initial assumption.
	\item Suppose the exact sequence (\ref{exact3-1}) does not split. Since there is a non-trivial extension 
	\[
	0\rightarrow q^*\oh_{\mathcal{M}}(-2)\otimes p^*\oh_{X}(1)\rightarrow q^*\oh_{\mathcal{M}}(-1)\otimes p^*\mathcal{S} \rightarrow \oh_{\mathcal{U}}\rightarrow 0,
	\]
	arising from the relative H-N filtration of $p^*\mathcal{S}$, we have $E_2\cong q^*\oh_{\mathcal{M}}(-1)\otimes p^*\mathcal{S}$, which implies that $p_*{E_2}$ and $R^1p_*{E_2}$ vanish. By projecting (\ref{exact3-2}) to $X$, we obtain $E\cong p_*q^*\oh_{\mathcal{M}}(2)\otimes \oh_{X}(-1)$. By \cite[Proposition 10.13]{ott}, this bundle is the homogeneous bundle $E_{2\lambda_2}(-1)$ and coincides with the tangent bundle $T_X(-1)$.\\
\end{itemize}
 %In conclusion, the proof of the theorem is complete.
\end{proof}

From now on, we call the kernel of any bundle epimorphism $\mathcal{L}^{\perp}\longrightarrow \mathcal{O}_{\mathbb{Q}^{n}}$ an \emph{orthogonal kernel bundle}, all such bundles will be denoted uniformly by $\mathcal{L}_{\mathrm{ker}}$. From the classification of uniform $3$-bundles on $\mathbb{Q}^3$, we observe several new features not present in the earlier literature.
%we observe several new phenomena that do not appear in the existing classification of low-rank uniform bundles over general Grassmannian manifolds $X$.

\begin{remark}
\begin{enumerate}
	\item In the literature, all known classified low-rank uniform vector bundles on a generalized Grassmannian are determined by their Chern classes (see \cite[Main Theorem]{EHS}, \cite[Theorem 1.1]{FLL}, \cite[Theorem 1]{guy} and \cite[Theorem 3.1]{MOS}). Nevertheless, it follows from the proof of Theorem \ref{uniform Q3} that every orthogonal kernel bundle $\mathcal{L}_{\mathrm{ker}}$ on $\mathbb{Q}^3$ carries the same Chern invariants: 
	\[
	 c_1(\mathcal{L}_{\mathrm{ker}})=-X_1, ~c_2(\mathcal{L}_{\mathrm{ker}})=X_1^2,~c_3(\mathcal{L}_{\mathrm{ker}})=-X_1^3.
	\]
This provides the first example of uniform bundles over a generalized Grassmannian which are not determined by their Chern classes.
	\item Let $\pi\colon Y\to X$ be a smooth projective morphism between smooth projective varieties. There exists a family of vector bundles on $Y$ whose restrictions to distinct fibers are non-isomorphic vector bundles. For instance, as observed from Claim \ref{claim:1}, for a null correlation bundle $N$ on $\mathbb{P}^3$, the restriction of $q^*N$ to a general $p$-fiber is isomorphic to $\oh^{\oplus 2}_{\mathbb{P}^1}$, while its restriction to a special
	$p$-fiber is isomorphic to $\oh_{\mathbb{P}^1}(1)\oplus\oh_{\mathbb{P}^1}(-1)$.
\end{enumerate} 
\end{remark}
 From Proposition \ref{prop1} and the definition of $\mathcal{L}_{\mathrm{ker}}$, we obtain the properties listed below. The arguments are analogous to those for Proposition \ref{prop1}.

\begin{proposition}\label{prop2}
	\begin{enumerate}
		\item $\mathcal{L}_{\mathrm{ker}}$ is stable and $\mathcal{L}_{\mathrm{ker}}^{\vee} $ is generated by global sections;
		\item $\mathcal{L}_{\mathrm{ker}}$ is uniform of splitting type $(0,\ldots,0,-1)$;
		\item The Chern classes of $\mathcal{L}_{\mathrm{ker}}$ satisfy $c_i(\mathcal{L}_{\mathrm{ker}})=(-1)^iX_1^i$ for any $1\le i\le n$;
		\item $H^1(\mathbb{Q}^{n},\mathcal{L}_{\mathrm{ker}})=1$ and $H^i(\mathbb{Q}^{n},\mathcal{L}_{\mathrm{ker}})=0$ for any $i\ne 1$;
		\item $H^0(\mathbb{Q}^{n},\mathcal{L}_{\mathrm{ker}}^{\vee})=n+1$ and $H^i(\mathbb{Q}^{n},\mathcal{L}_{\mathrm{ker}}^{\vee})=0$ for any $i>0$.
	\end{enumerate}
\end{proposition}

Taking the dual of an orthogonal kernel bundle, we see that it can be realized as the quotient bundle of a bundle monomorphism $a\colon\mathcal{O}_{\mathbb{Q}^{n}}\longrightarrow (\mathcal{L}^{\perp})^{\vee}$. Next, we describe the set of isomorphism classes of orthogonal kernel bundles on $\mathbb{Q}^{n}$. To this end, we must investigate, which bundle monomorphisms define isomorphic quotient bundles.

\begin{lemma}\label{lemma}
Let $a, a'\colon \mathcal{O}_{\mathbb{Q}^{n}}\longrightarrow (\mathcal{L}^{\perp})^{\vee}$ be two bundle monomorphisms, and let $E, E'$ be the associated quotient bundles. Then $E$ and $E'$ are isomorphic if and only if there exists a constant $c \in \mathbb{C} \setminus \{0\}$ such that $a' = c a$.
\end{lemma}
\begin{proof}
The quotients defined by $a$ and $ca$ are clearly isomorphic. Conversely, suppose
\[
a, a'\colon \mathcal{O}_{\mathbb{Q}^{n}}\longrightarrow (\mathcal{L}^{\perp})^{\vee}
\]
are bundle monomorphisms and let $\psi\colon E \rightarrow E'$ be an isomorphism of the quotients. %It is readily seen that every quotient bundle given by the above bundle monomorphism can be realized as the cohomology of a monad $0\rightarrow \mathcal{O}_{\mathbb{Q}^{n}}\rightarrow (\mathcal{L}^{\perp})^{\vee} \rightarrow 0\rightarrow 0$. 
By Proposition \ref{prop1},  $H^0(\mathcal{L}^{\perp})=H^1(\mathcal{L}^{\perp})=0$. It follows from \cite[Lemma 4.1.3]{OSS} that
every isomorphism $\psi\colon E \to E'$ can be lifted to a morphism between the corresponding short exact sequences, which gives
%an isomorphism of the associated monads, we thus get 
a commutative diagram
\[
\xymatrix{
	0\ar[r]& \mathcal{O}_{\mathbb{Q}^{n}}\ar[r]^a\ar[d]^\Psi& (\mathcal{L}^{\perp})^{\vee}\ar[d]^{\Psi'}\ar[r]& E\ar[d]^\psi\ar[r]&0\\
	0\ar[r]& \mathcal{O}_{\mathbb{Q}^{n}}\ar[r]^{a'}& (\mathcal{L}^{\perp})^{\vee}\ar[r] &E'\ar[r]&0.
}
\]
It is clear that $\mathcal{O}_{\mathbb{Q}^{n}}$ is simple. By Proposition \ref{prop1}, $(\mathcal{L}^{\perp})^{\vee}$ is stable and therefore simple, i.e., $\Psi$ and $\Psi'$ are homotheties:
\[
\Psi = \lambda \operatorname{id}_{\mathcal{O}_{\mathbb{Q}^{n}}},\quad \Psi' = \lambda' \operatorname{id}_{(\mathcal{L}^{\perp})^{\vee}}, ~\lambda,\lambda' \in \mathbb{C} \setminus \{0\}.
\]
Hence we have $a' = c a$ with $c = \lambda/\lambda'$.
\end{proof}

By Proposition \ref{prop1}, we see that $(\mathcal{L}^{\perp})^{\vee}$ is generated by global sections and $h^0(\mathbb{Q}^{n}, \mathcal{L}^{\perp})^{\vee})=n+2$. Since $\mathrm{rk}(\mathcal{L}^{\perp})=n+1>n$, every generic section $s$ of $(\mathcal{L}^{\perp})^{\vee}$ is nowhere-vanishing (cf. \cite[Corollary 5.5]{EH}), which is equivalent to saying that $s$ defines a trivial subbundle. Notice that this quotient bundle is exactly the dual of an orthogonal kernel bundle. Combining with Lemma \ref{lemma}, we immediately obtain the following deduction.

\begin{corollary}
There is a bijection from the set of isomorphism classes of orthogonal kernel bundles on $\mathbb{Q}^{n}$ onto a quasi-projective variety of dimension $n+1$.
\end{corollary}

	\subsection{Uniform $4$-bundles on $\mathbb{Q}^3$ with splitting type $(0,0,0,-1)$}
	In this section, we further classify uniform bundles of rank $4$ on $\mathbb{Q}^3$ with splitting type $(0,0,0,-1)$, which will be used in our classification of rank $5$ uniform bundles on $\mathbb{Q}^5$. We first prove an interesting result concerning rank $3$ vector bundles on $\mathbb P^3$. 
	
	\begin{theorem}\label{P3}
Let $G$ be a rank $3$ bundle on $\mathbb P^3$ with $c_1(G)=2$, $c_2(G)=2$ and $c_3(G)=0$. 
Suppose that $G|L\cong \oh_L(1)\oplus\oh_L(1)\oplus\oh_L$ for some line $L\subset\mathbb P^3$. Then $G$ is isomorphic to $\Omega_{\mathbb P^3}(2)$ or $N(1)\oplus \oh_{\mathbb P^3}$, where $N$ denotes a null correlation bundle.
	\end{theorem}
	\begin{proof}
  If $G$ is uniform or stable, then by the classification of uniform $3$-bundles on $\mathbb P^3$ and \cite[Corollary 3.4]{Sch}, $G$ can only be $\Omega_{\mathbb P^3}(2)$. It remains to treat bundles that are neither uniform nor stable. Let $G$ be such a bundle. Since $c_1(G)=2$, $G$ is not stable if and only if $G$ is not semistable. Let $F$ denote the maximal destabilizing saturated subsheaf of $G$. Then $F$ is a semistable reflexive sheaf and $G/F$ is torsion free. \\

(1) Assume $\mathrm{rk}(F)=2$. Since the slope $\mu(F)=\frac{c_1(F)}{2}>\mu(G)=\frac{c_1(G)}{3}=\frac{2}{3}$, we deduce $c_1(F)\ge 2$. Meanwhile, by the hypothesis $G|L\cong \oh_L(1)\oplus\oh_L(1)\oplus\oh_L$, we have $c_1(F)\le 2$. Consequently, $c_1(F)= 2$. It follows that $G/F$ is an ideal sheaf $I_{\Delta}$ with $\dim \Delta\le 1$.

If $\Delta$ contains a pure $1$-dimensional cycle $\Delta_1$ of degree $d$, then $c_1(I_\Delta)=0$ and $c_2(I_\Delta)=d$. From the Chern classes of $G$, we get $c_1(F)=2$ and $c_2(F)=2-d$. As $F$ is semistable, one has $c_1^2(F)\le 4c_2(F)$, which implies that $d=1$ and $F$ is not stable (for the Bogomolov inequality attains equality here). By \cite[Lemma 3.1]{Har}, $h^0(\mathbb P^3, F(-1))\ne 0$ and thus we have an exact sequence $0\rightarrow \oh_{\mathbb P^3}\rightarrow F(-1)\rightarrow I_C\rightarrow 0$, where $C$ is empty or a Cohen-Macaulay curve. Since $[C]=c_2(I_C)=c_2(F(-1))=0$, we obtain $I_C=\oh_{\mathbb P^3}$. Consequently, $F(-1)\cong \oh^{\oplus 2}_{\mathbb P^3}$, and therefore $F\cong \oh_{\mathbb P^3}(1)^{\oplus 2}$. Furthermore, combining the exact sequence 
\begin{align}\label{depth}
0\rightarrow F\rightarrow G\rightarrow I_\Delta\rightarrow 0
\end{align} 
with the Chern classes of $F$ and $G$, we deduce $c_3(I_\Delta)=-2$. This implies that $\Delta$ is not purely one-dimensional and hence carries embedded or isolated points. On the other hand, for any $p\in \mathbb P^3$, set $R=\mathcal{O}_{\mathbb P^3,p}$ and let $\operatorname{pd}(M)$ be the projective dimension of an $R$-module $M$. From the exact sequence (\ref{depth}), we obtain $\operatorname{pd}(I_{\Delta,p})\le \max\{\operatorname{pd}(G_p), \operatorname{pd}(F_p)+1\}=1$. By the Auslander‑Buchsbaum formula, $\operatorname{depth}(I_{\Delta,p})\ge 2$. The depth lemma yields $\operatorname{depth}(\oh_{\Delta,p})\ge 1$, so that $p$ cannot be an embedded or isolated point of $\Delta$, a contradiction. Therefore, we have $\dim\Delta =0$.\\
 %Take such a point $p$, and let $L$ be a line meeting $\Delta$ only at $p$.Then $\mathcal{T}or_1^{\mathcal{O}_X}( I_\Delta,\mathcal{O}_L)\cong \mathcal{T}or_2^{\mathcal{O}_X}( \mathcal{O}_\Delta,\mathcal{O}_L)$ is the skyscraper sheaf $\kappa(p)$.On the other hand, restricting the above exact sequence to $L$ yields the long exact sequence $0\rightarrow \mathcal{T}or_1^{\mathcal{O}_X}( I_\Delta,\mathcal{O}_L)\rightarrow F|_{L}\rightarrow  G|_{L}\rightarrow I_\Delta|_{L}\rightarrow 0$. Since $F|_{L}$ is a vector bundle on $L$, $\mathcal{T}or_1^{\mathcal{O}_X}( I_\Delta,\mathcal{O}_L)$must be torsion free, contrary to the above. 
%Combining the exact sequence $0\rightarrow F\rightarrow G\rightarrow I_{\Delta}\rightarrow 0$ with $c_1(G)=2$, we deduce that the restriction $I_{\Delta}|L$ is trivial for any line $L$. This yields $G|L\cong \oh_L(1)\oplus\oh_L(1)\oplus\oh_L$ for every line L, which contradicts the non-uniformity of $G$. Thus this case is impossible. \\

 By Serre duality, we easily get $\mathcal{E}xt^i(\oh_{\Delta},\oh_{\mathbb P^3})=0$ for $0\le i\le 2$ and thus $\mathcal{E}xt^0(I_{\Delta},\oh_{\mathbb P^3})=\oh_{\mathbb P^3}$, $\mathcal{E}xt^1(I_{\Delta},\oh_{\mathbb P^3})=0$. Then applying the $\mathcal{H}om(-,\oh_{\mathbb P^3})$ functor to the exact sequence $0\rightarrow F\rightarrow G\rightarrow I_{\Delta}\rightarrow 0$, we obtain a new exact sequence 
\[0\rightarrow \oh_{\mathbb P^3}\rightarrow G^{\vee}\rightarrow F^{\vee}\rightarrow 0. \]
The condition $c_3(G)=0$ forces $c_3(F^{\vee})=0$, so $F^{\vee}$ is a vector bundle with $c_1(F^{\vee})=c_1(G^{\vee})=-2$, $c_2(F^{\vee})=c_2(G^{\vee})=2$. Since $F^{\vee}$ is semistable, $F^{\vee}(1)$ is a null correlation bundle $N$. So $F^{\vee}$ is $N(-1)$. Since the group $\mathrm{Ext}^1(N(-1),\oh_{\mathbb P^3})$ vanishes, we have $G^{\vee}\cong N(-1)\oplus \oh_{\mathbb P^3}$ and thus $G\cong N(1)\oplus \oh_{\mathbb P^3}$.\\
 
\iffalse
and $\deg\Delta=l$, which yields $c_1(I_\Delta)=c_2(I_\Delta)=0$ and $c_3(I_\Delta)=2l$. Then $c_1(F)=c_2(F)=2$
and $c_3(F)=2l$. We shall discuss three cases separately.
\begin{itemize}
	\item Case 1. $F$ is stable and $F=N(1)$, where $N$ a null correlation bundle. 
	
	Then $c_3(F)=0$, which implies $I_\Delta=\oh_{\mathbb P^3}$. Since the group $\mathrm{Ext}^1(\oh_{\mathbb P^3},N(1))$ vanishes, we have $G\cong N(1)\oplus \oh_{\mathbb P^3}$.
	\item Case 2. $F$ is stable but $F$ is not $N(1)$. 
	
	Let's consider the stable reflexive sheaf $F(-1)$,whose Chern classes are $c_1(F(-1))=0$ and $c_2(F(-1))=1$. As F is not null correlation bundle, it follows from \cite[Remark 3,3,1]{Har1} that there exists a general plane $\mathbb P^2$ (disjoint from the singular set of $F(-1)$) such that $F(-1)|\mathbb P^2$ is still stable. However, there is no stable $2$-bundle on $\mathbb P^2$ with $c_1=0$ and $c_2=1$, which gives a contradiction. Hence this case is impossible.
	
	\item Case 3. $F$ is not stable.
\end{itemize}
In conclusion, we prove that $G\cong N(1)\oplus \oh_{\mathbb P^3}$ whenever $\mathrm{rk}(F)=2$. We further eliminate the case $\mathrm{rk}(F)=1$, which finishes the proof of the theorem.
\fi

(2) Assume $\mathrm{rk}(F)=1$. Similar to the argument in (1), one can deduce that $c_1(F)=1$. Suppose $0\subset F\subset F'\subset G$ is the H-N filtration of $G$. Then the sheaves $F$, $F'/F$ and $G/F'$ are of rank $1$, and $c_1(F)>c_1(F'/F)>c_1(G/F')$. Since $c_1(F)=1$, we would have $c_1(F'/F)\le0$, $c_1(G/F')\le -1$, contrary to the hypothesis $c_1(G)=2$. Therefore, the H-N filtration of $G$ has exactly two terms. Let $\widetilde{F}$ be the dual sheaf of $G/F$. Then $\widetilde{F}$ is the maximal destabilizing saturated subsheaf of $G^{\vee}$ with $G^{\vee}/\widetilde{F}=I_{\Gamma}(-1)$, where $\dim \Gamma\le 1$. 
We intend to exclude this scenario. %Following the same line of reasoning as in (1), we present a concise proof for ease of reference. 

If $\Gamma$ contains a pure $1$-dimensional cycle $\Gamma_1$ of degree $m$, one computes $c_1(\widetilde{F})=-1$ and $c_2(\widetilde{F})=1-m$. Since $\widetilde{F}$ is semistable, the inequality $c_1^2(\widetilde{F})\le 4c_2(\widetilde{F})$ gives $m=0$, which immediately yields a contradiction. If $\dim \Gamma=0$, then from 
Serre duality, we would have 
 $\mathcal{E}xt^0(I_{\Gamma}(-1),\oh_{\mathbb P^3})=\oh_{\mathbb P^3}(1)$, $\mathcal{E}xt^1(I_{\Gamma}(-1),\oh_{\mathbb P^3})=0$. Then applying the $\mathcal{H}om(-,\oh_{\mathbb P^3})$ functor to the exact sequence $0\rightarrow \widetilde{F}\rightarrow G^{\vee}\rightarrow I_{\Gamma}(-1)\rightarrow 0$, we obtain a new exact sequence 
\[0\rightarrow \oh_{\mathbb P^3}(1)\rightarrow G\rightarrow \widetilde{F}^{\vee}\rightarrow 0. \]
The conditions $c_1(G)=2$, $c_2(G)=2$ and $c_3(G)=0$ force $c_3(\widetilde{F}^{\vee})=-1$. Nevertheless, since $\widetilde{F}^{\vee}$ is reflexive, this is impossible by \cite[Proposition 2.6]{Har}. Then we are done.
\end{proof}

%Recall $\mathcal{L}^{\perp}$ is the $\mathcal{Q}$-orthogonal complement of $\mathcal{O}_{\mathbb{Q}^{n}}(-1)$. 
\begin{theorem}\label{uniform Q3 2}
Let $E$ be a uniform bundle of rank $4$ on $\mathbb{Q}^{3}$ of splitting type $(0,0,0,-1)$. Then $E$ either splits, or is isomorphic to
\[
\mathcal{S}(-1)\oplus \mathcal{O}^{\oplus 2}_{\mathbb{Q}^{3}},~ \mathcal{L}^{\perp},~\text{or}~\mathcal{L}_{\mathrm{ker}}\oplus \mathcal{O}_{\mathbb{Q}^{3}}.
\]
\end{theorem}
\begin{proof}
	%In order to prove this theorem 
	We adopt the notation established in Section \ref{sec1}, under which the relative H-N filtration of $p^*E$ induces an exact sequence
\begin{align}\label{exact3}
	0\rightarrow q^*G \rightarrow p^*E \rightarrow q^*L\otimes p^*\oh_{\mathbb{Q}^3}(-1) \rightarrow 0,
\end{align} where the rank of $G$ is $3$ and $L$ is a line bundle. 
By the same argument as in Case I of Theorem \ref{uniform Q3}, we conclude that $L$ is $\oh_{\mathcal{M}}$, $\oh_{\mathcal{M}}(1)$ or $\oh_{\mathcal{M}}(2)$.\\

If $L$ is $\oh_{\mathcal{M}}$, then $E$ splits by the same reasoning as in Case I of Theorem \ref{uniform Q3}. If $L$ is $\oh_{\mathcal{M}}(1)$, then
$q^*L\otimes p^*\oh_{\mathbb{Q}^3}(-1)|_{p^{-1}(x)}\cong \oh_{\mathbb{P}^1}(1)$ and $q^*G|_{p^{-1}(x)}\cong\oh^{\oplus 2}_{\mathbb{P}^1}\oplus\oh_{\mathbb{P}^1}(-1)$. Therefore $p_*q^*G$ is a rank $2$ bundle and $R^1p_*q^*G$ vanishes by the base change theorem. By projecting (\ref{exact3}) to $\mathbb{Q}^{3}$, we obtain an exact sequence $0\rightarrow p_*q^*G \rightarrow E \rightarrow p_*q^*\oh_{\mathcal{M}}(1)\otimes \oh_{\mathbb{Q}^{3}}(-1)(=\mathcal{S}(-1))  \rightarrow 0$. Since $\mathcal{S}(-1)$ is uniform of splitting type $(0,-1)$, $p_*q^*G$ is uniform of splitting type $(0,0)$. Consequently, $p_*q^*G$ is trivial and we conclude that 
 $E\cong\mathcal{S}(-1)\oplus \mathcal{O}_{\mathbb{Q}^{3}}^{\oplus 2}$. Its Chern classes are $c_1(E)=-X_1$, $c_2(E)=\frac{1}{2}X_1^2$, $c_3(E)=c_4(E)=0$.\\

If $L$ is $\oh_{\mathcal{M}}(2)$, one deduces readily that $c_1(E)=-X_1, ~c_2(E)=X_1^2 ,~c_3(E)=-X_1^3$. It follows that 
$G$ is a rank $3$ bundle on $\mathbb{P}^3(\cong \mathcal{M})$ with $c_1(G)=-2$, $c_2(G)=2$ and $c_3(G)=0$. 
 Hence $q^*G|_{p^{-1}(x)}$ is either isomorphic to $\oh_{\mathbb{P}^1}(-1)^{\oplus 2}\oplus \oh_{\mathbb{P}^1}$ or  $\oh_{\mathbb{P}^1}^{\oplus 2}\oplus \oh_{\mathbb{P}^1}(-2)$. We shall proceed with our argument in three separate cases.

\begin{enumerate}
	\item For any $x\in \mathbb{Q}^3$, $q^*G|_{p^{-1}(x)}\cong \oh_{\mathbb{P}^1}^{\oplus 2}\oplus \oh_{\mathbb{P}^1}(-2)$.
	
	By the base change theorem, $p_*q^*G$ is a rank $2$ bundle and 
	 $R^1p_*q^*G$ is a line bundle. Write $R^1p_*q^*G=\oh_{\mathbb{Q}^{3}}(a)$. Applying $R^ip_*$ to  (\ref{exact3}), we obtain $0\rightarrow p_*q^*G \rightarrow E \rightarrow T_{\mathbb{Q}^{3}}(-1) \rightarrow \oh_{\mathbb{Q}^{3}}(a)\rightarrow 0$. Then $c_1(p_*q^*G)=a-1$. 
	 Since the splitting type of $E$ is $(0,0,0,-1)$, we have $a-1\le 0$. On the other hand, the stability of $T_{\mathbb{Q}^{3}}(-1)$ implies that $a\ge 1$. This forces $a=1$ and hence $p_*q^*G$ is trivial, as it is uniform of splitting type $(0,0)$. Thus we have $c(\oh^{\oplus 2}_{\mathbb{Q}^{3}})c(T_{\mathbb{Q}^{3}}(-1))=c(E)c(\oh_{\mathbb{Q}^{3}}(1))$, which contradicts the Chern classes of $E$ computed above. This case is therefore impossible.
	 
	 \item For any $x\in \mathbb{Q}^3$, $q^*G|_{p^{-1}(x)}\cong\oh_{\mathbb{P}^1}(-1)^{\oplus 2}\oplus \oh_{\mathbb{P}^1}$.
	 
	 By the base change theorem, $p_*q^*G$ is a line bundle and $R^1p_*q^*G$ vanishes. By projecting (\ref{exact3}) to $\mathbb{Q}^{3}$, we have $0\rightarrow p_*q^*G \rightarrow E \rightarrow T_{\mathbb{Q}^{3}}(-1) \rightarrow 0$. Since $c_1(E)=-X_1$ and $c_1(T_{\mathbb{Q}^{3}}(-1))=0$, we have $p_*q^*G\cong \oh_{\mathbb{Q}^{3}}(-1)$. Then $E$ is isomorphic to $\mathcal{L}^{\perp}$ for $\dim \mathrm{Ext}^1(T_{\mathbb{Q}^{3}}(-1),\oh_{\mathbb{Q}^3}(-1))=1$ and the splitting type of $E$ is $(0,0,0,-1)$.
	 
	 \item There exists a point $x$ for which $q^*G|_{p^{-1}(x)}\cong \oh_{\mathbb{P}^1}^{\oplus 2}\oplus \oh_{\mathbb{P}^1}(-2)$, and a point $y$ for which $q^*G|_{p^{-1}(y)}\cong\oh_{\mathbb{P}^1}(-1)^{\oplus 2}\oplus \oh_{\mathbb{P}^1}$.
	 
	 Thus $G$ is not uniform and there exists a line $L\subset\mathbb{P}^3$ satisfying $G|L\cong \oh_L(-1)\oplus\oh_L(-1)\oplus\oh_L$. Since $c_1(G)=-2$, $c_2(G)=2$ and $c_3(G)=0$, we have $G\cong N(-1)\oplus \oh_{\mathbb P^3}$ 
	 by Theorem \ref{P3}. Consequently, $\oh_{\mathcal{U}}$ is a subbundle of $q^*G$, and it follows from (\ref{exact3}) that $\oh_{\mathcal{U}}$ is also a subbundle of $p^*E$. Taking $p_{*}$, we immediately see that $\oh_{\mathbb{Q}^3}$ is a subbundle of $E$. Denote the quotient bundle by $F$, which is a uniform bundle on $\mathbb{Q}^3$ with splitting type $(0,0,-1)$. Given its Chern classes being $c_1(F)=-X_1, ~c_2(F)=X_1^2 ,~c_3(F)=-X_1^3$,
	Theorem \ref{uniform Q3} implies that $F$ is an orthogonal kernel bundle, i.e., $F\cong\mathcal{L}_{ker}$.
 By Proposition \ref{prop2}, $\dim \mathrm{Ext}^1(\mathcal{L}_{ker}, \oh_{\mathbb Q^3})=0$, so $E$ is isomorphic to $\mathcal{L}_{ker}\oplus \mathcal{O}_{\mathbb{Q}^{3}}$.
\end{enumerate}
\end{proof}
As an immediate consequence of Theorem \ref{uniform Q3 2}, we obtain the following corollary.
	\begin{corollary}\label{uniform Q3 3}
		Let $E$ be a uniform bundle of rank $4$ on $\mathbb{Q}^{3}$ of splitting type $(0,0,0,-1)$ and Chern classes  
		\[
		c_1(E)=-X_1, ~c_2(E)=X_1^2 ,~\text{and}~c_3(E)=-X_1^3.
		\]
		Then $E$ is isomorphic to either $\mathcal{L}^{\perp}$ or $\mathcal{L}_{\mathrm{ker}}\oplus \mathcal{O}_{\mathbb{Q}^{3}}$.
	\end{corollary}
	\subsection{Uniform $5$-bundles on $\mathbb{Q}^5$}
	Throughout this section, let $E$ be a uniform bundle on $\mathbb{Q}^5$ of rank $5$ with splitting type $(a_1,\ldots,a_5)$. By Lemma \ref{lemma1}, if there is an integer $i~(1\le i\le 4)$ such that $a_i-a_{i+1}\geq 2$, then $E$ has a uniform subbundle of rank at most $4$. 
	According to \cite[Theorem 4.1]{KS} and \cite[Theorem 4.3]{FLL}, such subbundles either split, or are $\mathcal{S}(a)~(a\in \mathbb{Z})$. Hence $E$ either splits or is of the form $\mathcal{S}(a)\oplus \mathcal{O}_{\mathbb{Q}^{5}}(b)$ with $a,b\in\mathbb{Z}$. Thus we may assume $a_i-a_{i+1}\le 1$ for every $1\le i\le 4$. Moreover, by \cite[Proposition 1.2]{AW} and \cite[Proposition 3.5]{FLL}, 
it suffices to consider the following seven splitting types 
\[
(0,0,0,-1,-1),~(1,1,0,0,-1),~(1,0,0,0,-1),~(0,0,0,0,-1),~(2,1,0,0,0),~(2,1,0,0,-1),~(2,2,1,0,0).
\]
up to dual and tensoring by a line bundle. We then discuss each case separately. For later use, we recall the standard diagram and the corresponding Chow rings (cf. \cite[Lemma 2.1]{FLL}). 
		\begin{align*}
		\xymatrix{
			\mathcal{U}=B_3/P_{1,2}\ar[d]^{p}   \ar[r]^-{q} & \mathcal{M}=B_3/P_2\\
			X=\mathbb{Q}^5.
		}
	\end{align*}
	
	 %\begin{equation*}
		\begin{align*}
			&H^\bullet(X,\mathbb{Q})=\mathbb{Q}[X_1]/(X_1^6),~H^{\bullet}(\mathcal{U},\mathbb{Q})=\mathbb{Q}[X_1,X_2]/I,\\
			&H^{\bullet}(\mathcal{M},\mathbb{Q})=\mathbb{Q}[X_1+X_2,X_1X_2]/I,~I=(\Sigma_d(X_1^2,X_2^2)_{2\leq d\leq 3}).
		\end{align*}
	%\end{equation*}
The fiber of $q$ is isomorphic to a line in $X$ and the fiber of $p$ is isomorphic to $\mathbb{Q}^3$. For $1\le t\le 5$, we denote the $t$-th Chern class $c_t(p^*E)$ by $\mu_tX_1^t \pmod{I}$ for some $\mu_t\in \mathbb{Q}$. In the proofs below, all Chern class computations are carried out via Whitney's formula, and the resulting polynomial systems are solved using Mathematica (routine algebraic details omitted).
	
\begin{proposition}\label{prop3-2}
	Let $E$ be a uniform $5$-bundle on $\mathbb{Q}^5$ of splitting type $(0,0,0,-1,-1)$. Then $E$ either splits or is isomorphic to $\mathcal{S}(-1)\oplus \oh_{\mathbb{Q}^5}$.
\end{proposition}
\begin{proof}
Let $0\rightarrow E_1 \rightarrow p^*E \rightarrow E_2 \rightarrow 0$ be the exact sequence induced by the relative H-N filtration. There are bundles $G_1$ of rank $3$ and $G_2$ of rank $2$ on $\mathcal{M}$ satisfying $E_1=q^*G_1$ and $E_2=q^*G_2\otimes p^*\oh_X(-1)$. We write the Chern classes as in the following equations
\begin{equation*}
	\begin{aligned}
		&c_1(E_1)=a_1(X_1+X_2),~c_2(E_1)=a_2X_1X_2+a_2'(X_1^2+X_2^2),\\
		&c_3(E_1)=a_3X_1X_2(X_1+X_2)+a_3'(X_1^3+X_2^3);\\
		&c_1(E_2)=b_1(X_1+X_2)-2X_1,~c_2(E_2)=b_2X_1X_2+b_2'(X_1^2+X_2^2)-b_1(X_1+X_2)X_1+X_1^2.
	\end{aligned}
\end{equation*}
Note that $\mathrm{Pic}(\mathcal{M})=\mathbb{Z}.\oh_{\mathcal{M}}(1)$ and $c_1(\oh_{\mathcal{M}}(1))=X_1+X_2$, hence $a_1,b_1\in\mathbb Z$. From $c_1(p^*E)=-2X_1$, we have $b_1=-a_1$. Furthermore, additional equations involving $a_i,b_i,a_i',b_i'$ can be deduced from the constraint $c_t(p^*E)\equiv\mu_tX_1^t\pmod{I}$. A direct calculation shows that $a_1\in \{0,-1\}$.

If $a_1$ is $0$, then $b_1=0$. Thus $c_1(E_2|_{p^{-1}(x)})=0$ for any $x\in X$. By \cite[Lemma 3,1]{FLL}, $E_2|_{p^{-1}(x)}$ is trivial, which implies that $E_1|_{p^{-1}(x)}$ is also trivial. Arguing as in Theorem \ref{uniform Q3}, $p_*E_1$ and $p_*E_2$ are uniform bundles of splitting types $(0,0,0)$ and $(-1,-1)$ respectively, hence they split. Consequently, $E$ splits. \\

If $a_1$ is $-1$, then $b_1=1$ and it follows that
$\det(E_2|_{p^{-1}(x)})$ is isomorphic to $\oh_{\mathbb{Q}^3}(1)$ for each $x\in X$. In fact, combining the relations among $a_i,b_i,a_i',b_i'$, we further deduce  $a_2=b_2=a_2'=b_2'=\frac{1}{2}$ and $a_3=a_3'=0$. On the other hand, by the universal property of Grassmannians, for any $x\in X$,
there is a morphism $\psi_x\colon p^{-1}(x)(\cong \mathbb{Q}^3)\rightarrow Gr(3,5)$ induced by the relative H-N filtration such that $E_1|_{p^{-1}(x)}\cong \psi^*U$ and $E_2|_{p^{-1}(x)}\cong \psi_x^*Q$, where $U$ and $Q$ are the universal subbundle and quotient bundle on $Gr(3,5)$. Since $\det(E_2|_{p^{-1}(x)})\cong\oh_{\mathbb{Q}^3}(1)$, we have 
$\psi_x^*\oh_{Gr(3,5)}(1)\cong \psi_x^*(\det Q)\cong \det(E_2|_{p^{-1}(x)})\cong \oh_{\mathbb{Q}^3}(1)$. So $\psi_x$ maps lines on $\mathbb{Q}^3$ isomorphically to lines on $Gr(3,5)$. Therefore, as the pull-back of uniform bundles on $Gr(3,5)$, $E_1|_{p^{-1}(x)}$ and $E_2|_{p^{-1}(x)}$ are uniform bundles on $\mathbb{Q}^3$. %Furthermore, the relations $a_2=b_2=a_2'=b_2'=\frac{1}{2}$ and $a_3=a_3'=0$ imply $c_1(E_1|_{p^{-1}(x)})=-X_1$
Note that the Chern classes of $E_1|_{p^{-1}(x)}$ (resp. $E_2|_{p^{-1}(x)}$) are the same as those of $E_{\lambda_3}(-1)\oplus\oh_{\mathbb{Q}^3}$ (resp. $E_{\lambda_3}$), where $E_{\lambda_3}$ is the irreducible homogeneous bundle with highest weight $\lambda_3$ on $\mathbb{Q}^3$. Consequently, for every $x\in X$, $E_1|_{p^{-1}(x)}\cong E_{\lambda_3}(-1)\oplus\oh_{\mathbb{Q}^3}$ and $E_2|_{p^{-1}(x)}\cong E_{\lambda_3}$ by \cite[Theorem 3.1]{fri} and Theorem \ref{uniform Q3}. \\

Let $\pi\colon B_3/P_{1,2,3}\rightarrow \mathcal{U}(=B_3/P_{1,2})$ denote the natural projection. Then $\pi_{*}L_{\lambda_3}$ is an irreducible homogeneous $2$-bundle, where $L_{\lambda_3}$ is the line bundle on $B_3/P_{1,2,3}$ with weight $\lambda_3$. For each $x\in X$, we have $(\pi_{*}L_{\lambda_3})|_{p^{-1}(x)}\cong E_{\lambda_3}\cong E_2|_{p^{-1}(x)}$ and $E_{\lambda_3}$ is simple.
By \cite[Lemma 4.1]{FLL}, this yields an isomorphism
$E_2\cong \pi_{*}L_{\lambda_3}\otimes p^*\oh_X(t)$ for some integer $t$. According to \cite[Proposition 10.13]{ott}, the sheaf $p_*(\pi_{*}L_{\lambda_3})$ is isomorphic to 
the spinor bundle $\mathcal{S}$ on $X$. Hence $p_{*}E_2\cong\mathcal{S}(t)$.

%the sheaf $p_*(\pi_{*}L_{\lambda_3})$ is the homogeneous bundle $E_{\lambda_3}(-1)$. Here we use identical symbols without confusion. Indeed, $E_{\lambda_3}$ is the homogeneous bundle $\mathcal{S}(-1)$ on $\mathbb{Q}^5$. 
From $h^0(\mathbb{Q}^3,E_{\lambda_3}(-1)\oplus\oh_{\mathbb{Q}^3})=1$, $h^1(\mathbb{Q}^3,E_{\lambda_3}(-1)\oplus\oh_{\mathbb{Q}^3})=0$, we deduce that $p_{*}E_1$ is a line bundle $\oh_{X}(s)$ and $R^1p_{*}E_1=0$. By projection the sequence $0\rightarrow E_1 \rightarrow p^*E \rightarrow E_2 \rightarrow 0$ onto $X$, we obtain an exact sequence 
\[
0\rightarrow \oh_{X}(s) \rightarrow E \rightarrow \mathcal{S}(t) \rightarrow 0.
\]
 Since $\mathcal{S}$ is ACM, $\mathrm{Ext}^1(\mathcal{S}(t),\oh_{X}(s))$ vanishes, thus we have $E\cong \mathcal{S}(t)\oplus \oh_{X}(s)$. The splitting type of $E$ then forces $t=-1$ and $s=0$. Then we are done.

 %Because $c_1(p_{*}E_1)=c_1(E)-c_1(\mathcal{S}(-1))=0$,  $p_{*}E_1$ is a trivial bundle. As the group $\mathrm{Ext}^1(\mathcal{S}(-1),\oh_{\mathbb{Q}^5})$ vanishes, we immediately obtain $E\cong \mathcal{S}(-1)\oplus \oh_{\mathbb{Q}^5}$.
\end{proof}	
	
	\begin{proposition}\label{prop3-3}
		Let $E$ be a uniform $5$-bundle on $\mathbb{Q}^5$ of splitting type $(1,1,0,0,-1)$. Then $E$ either splits or is isomorphic to $\mathcal{S}\oplus \oh_{\mathbb{Q}^5}(-1)$.
	\end{proposition}
	\begin{proof}
		Let
		\begin{align}
			&0\rightarrow E_1(=q^*G_1\otimes p^*\oh_{X}(1))\rightarrow  E_2\rightarrow E_2/E_1(=q^*{G_2})\rightarrow 0\notag,\\
			&0\rightarrow E_2\rightarrow  p^*E\rightarrow E_3(=q^*{G_3}\otimes p^*\oh_{X}(-1))\rightarrow 0\label{exact4}, 
		\end{align} 
be the exact sequences induced by the relative H-N filtration,		
where every $G_1$, $G_2$ are rank $2$ bundles and $G_3$ is a line bundle. We denote the Chern classes by
	\begin{equation*}
		\begin{aligned}
			&c_1(E_1)=a_1(X_1+X_2)+2X_1,~c_2(E_1)=b_1(X_1^2+X_2^2)+b_2X_1X_2+a_1X_1(X_1+X_2)+X_1^2,\\
			&c_1(E_2/E_1)=a_2(X_1+X_2),~c_2(E_2/E_1)=b_1'(X_1^2+X_2^2)+b_2'X_1X_2.
		\end{aligned}
	\end{equation*}
	Then $c_1(E_3)=-(a_1+a_2)(X_1+X_2)-X_1$. A direct computation with the condition $c_t(p^*E)\equiv\mu_tX_1^t \pmod{I}$ gives the admissible pairs $(a_1,a_2)$ are $(0,0)$, $(-1,1)$, or $(-1,0)$.\\
	
	If $(a_1,a_2)$ is $(0,0)$ or $(-1,1)$, then $a_1+a_2=0$ and hence $G_3=\oh_{\mathcal{M}}$. It follows that $E_3|_{p^{-1}(x)}$
	and $E_2|_{p^{-1}(x)}$ are trivial. Then $p_{*}E_3$ is $\oh_X(-1)$ and $p_{*}E_2$ is a uniform bundle of splitting type $(1,1,0,0)$. By \cite[Theorem 4.3]{FLL}, $p_{*}E_2$ either splits or is the spinor bundle $\mathcal{S}$. Since $E$ is an extension of $p_{*}E_2$ and $p_{*}E_3$, $E$ either splits or is $\mathcal{S}\oplus \oh_{\mathbb{Q}^5}(-1)$.\\
	
	If $(a_1,a_2)$ is $(-1,0)$, then $G_3=\oh_{\mathcal{M}}(1)$. It follows that
$E_3|_{p^{-1}(x)}\cong \oh_{\mathbb{Q}^3}(1)$ and hence $E_2|_{p^{-1}(x)}$ is a uniform bundle with splitting type $(0,0,0,-1)$ on $\mathbb{Q}^3$. Note that the Chern classes of $E_2|_{p^{-1}(x)}$ are the same as those of $\mathcal{L}^{\perp}$. By Corollary \ref{uniform Q3 3}, $E_2|_{p^{-1}(x)}$ is isomorphic to either $\mathcal{L}^{\perp}$ or 
$\mathcal{L}_{\mathrm{ker}}\oplus \mathcal{O}_{\mathbb{Q}^{3}}$. Next, we elaborate three subcases to prove that this scenario cannot happen.

\textbf{Case a.} For any $x\in X$, $E_2|_{p^{-1}(x)}$ is isomorphic to  $\mathcal{L}^{\perp}$. 

By Proposition \ref{prop1}, $h^0(\mathbb{Q}^{3},\mathcal{L}^{\perp})=h^1(\mathbb{Q}^{3},\mathcal{L}^{\perp})=0$, %(see Proposition \ref{prop1}), 
so the direct images $p_{*}E_2$ and $R^1p_{*}E_2$ vanish. By performing $R^ip_*$ to the exact sequence (\ref{exact4}), we would have $E\cong p_{*}E_3\cong T_X(-1)$, contrary to the splitting type of $E$.

\textbf{Case b.} For any $x\in X$, $E_2|_{p^{-1}(x)}$ is isomorphic to  $\mathcal{L}_{\mathrm{ker}}\oplus \mathcal{O}_{\mathbb{Q}^{3}}$. 

By Proposition \ref{prop2}, $h^0(\mathbb{Q}^{3},\mathcal{L}_{\mathrm{ker}}\oplus \mathcal{O}_{\mathbb{Q}^{3}})=h^1(\mathbb{Q}^{3},\mathcal{L}_{\mathrm{ker}}\oplus \mathcal{O}_{\mathbb{Q}^{3}})=1$, %(see Proposition \ref{prop2}),
 so $p_{*}E_2$ and $R^1p_{*}E_2$ are line bundles by the base change theorem. Suppose that $p_{*}E_2=\oh_X(a)$ and $R^1p_{*}E_2=\oh_X(b)$.
By performing $R^ip_*$ to the exact sequence (\ref{exact4}), we would have an exact sequence 
\[
0\rightarrow \oh_X(a) \rightarrow E \rightarrow T_X(-1)\rightarrow\oh_X(b)\rightarrow 0. 
\]
Since $c_1(E)=X_1$ and $c_1(T_X(-1))=0$, we would have $a=b+1$. Note that $T_X(-1)$ is stable, which forces $b\ge 1$ and hence $a\ge 2$. This immediately contradicts the splitting type of $E$.

 \textbf{Case c.} There exists a point $x$ such that $E_2|_{p^{-1}(x)}\cong \mathcal{L}^{\perp}$ and a point $y$ such that $E_2|_{p^{-1}(y)}\cong \mathcal{L}_{\mathrm{ker}}\oplus \mathcal{O}_{\mathbb{Q}^{3}}$.
 
 Let $S_E:=\{y\in X\mid E_2|_{p^{-1}(y)}\cong\mathcal{L}_{\mathrm{ker}}\oplus \mathcal{O}_{\mathbb{Q}^{3}}\}$. Then the set $S_E$ is closed in $X$ by the semicontinuity theorem. 
 Since $p_*{E_2}$ is torsion free and zero outside of $S_E$, we conclude that $p_*{E_2}$ vanishes, which yields the following short exact sequence:
 \[
 	0\rightarrow E \rightarrow T_X(-1) \rightarrow R^1p_*{E_2}\rightarrow 0,
 \]
 where $R^1p_*{E_2}$ is zero or a torsion sheaf. Since $c_1(E)=X_1$ and $c_1(T_X(-1))=0$, we would have $c_1(R^1p_*{E_2})=-X_1$, which is absurd.
	\end{proof}	
		\begin{proposition}\label{prop3-4}
		Let $E$ be a uniform $5$-bundle on $\mathbb{Q}^5$ of splitting type $(1,0,0,0,-1)$. Then $E$ either splits or is isomorphic to the bundle $T_{\mathbb{Q}^5}(-1)$.
	\end{proposition}
	\begin{proof}
	The relative H-N filtration of $p^*E$ takes the same form as in Proposition \ref{prop3-3}, while in this case $G_1$ and $G_3$ are line bundles and $G_2$ is a rank $3$ bundle. We hereby adopt the notations used therein. Suppose that 
		\begin{equation*}
		\begin{aligned}
			&c_1(E_1)=a_1(X_1+X_2)+X_1,~c_1(E_2/E_1)=a_2(X_1+X_2),\\
			&c_2(E_2/E_1)=b_1'(X_1^2+X_2^2)+b_2'X_1X_2,~c_3(E_2/E_1)=t_1(X_1^3+X_2^3)+t_2X_1X_2(X_1+X_2).
		\end{aligned}
	\end{equation*}
	Then $c_1(E_3)=-(a_1+a_2)(X_1+X_2)-X_1$. %We employ Mathematica and Whitney's formula to compute the Chern classes of $p^*E$. 
	Given that $c_t(p^*E)\equiv\mu_tX_1^t \pmod{I}$, the admissible pairs for $(a_1,a_2)$ are $(0,0)$ or $(-1,0)$.\\
	
	If $(a_1,a_2)$ is $(0,0)$, then $G_3=\oh_{\mathcal{M}}$. It follows that $E_3|_{p^{-1}(x)}$
	and $E_2|_{p^{-1}(x)}$ are trivial. Then $p_{*}E_3$ is $\oh_X(-1)$ and $p_{*}E_2$ is a uniform bundle of splitting type $(1,0,0,0)$. It follows from \cite[Theorem 4.3]{FLL} that $p_{*}E_2$ splits. Since $E$ is an extension of $p_{*}E_2$ and $p_{*}E_3$, $E$ splits as well.\\
	
	If $(a_1,a_2)$ is $(-1,0)$, we get $b_1'=1$ and $b_2'=t_1=t_2=0$ by calculation. So the Chern classes of $E$ are $c_i(E)=0$ for odd $i$ and $c_i(E)=X_1^i$ for even $i$.
	%\[c_i(E)=0~\text{for odd}~i, \quad ~\text{and}~c_i(E)=X_1^i~\text{for even}~i.\]
	 %we get $G_3=\oh_{\mathcal{M}}(1)$. It follows that$E_3|_{p^{-1}(x)}\cong \oh_{\mathbb{Q}^3}(1)$ and hence $E_2|_{p^{-1}(x)}$ is a uniform bundle with splitting type $(0,0,0,-1)$ on $\mathbb{Q}^3$. Note that the Chern classes of $E_2|_{p^{-1}(x)}$ are the same as those of $\mathcal{L}^{\perp}$. By Corollary \ref{uniform Q3 3}, $E_2|_{p^{-1}(x)}$ is isomorphic to either $\mathcal{L}^{\perp}$ or 
	Moreover, by the same argument as in the proof of Proposition \ref{prop3-3} for the case $(a_1,a_2)=(-1,0)$, we have $E_2|_{p^{-1}(x)}\cong \mathcal{L}^{\perp}$ or $\mathcal{L}_{\mathrm{ker}}\oplus \mathcal{O}_{\mathbb{Q}^{3}}$.\\

\begin{claim}\phantomsection\label{claim3-4} There necessarily exists a point $x\in X$ such that $E_2|_{p^{-1}(x)}$ is isomorphic to $\mathcal{L}^{\perp}$.
\end{claim}

Proof. Assume that for any $x\in X$, $E_2|_{p^{-1}(x)}\cong \mathcal{L}_{\mathrm{ker}}\oplus \mathcal{O}_{\mathbb{Q}^{3}}$, by the base change theorem, $p_*{E_2}$ and $R^1p_*{E_2}$ are both line bundles. Denote $p_*{E_2}= \mathcal{O}_X(a)$.
By projecting
(\ref{exact4}) to $X$, we get an exact sequence
$0\rightarrow \mathcal{O}_X(a) \rightarrow E \rightarrow T_X(-1) \rightarrow R^1p_*{E_2}\rightarrow0$. Since 
$c_1(E)=c_1(T_X(-1))=0$, we obtain $R^1p_*{E_2}\cong \mathcal{O}_X(a)$. Combining the splitting type of $E$ and the stability of $T_X(-1)$, we deduce $a=1$, hence $p_*{E_2}=R^1p_*{E_2}=\oh_X(1)$. Let $F:=\ker(T_X(-1) \rightarrow \oh_X(1))$. 
From the splitting types of $E$, one easily checks that $F$ is a uniform $4$-bundle on $X$ of splitting type $(0,0,0,-1)$. Thus $F$ splits, which in turn forces $T_X(-1)$ to split. This contradicts the stability of $T_X(-1)$, so such a case cannot occur.\\

By Claim \ref{claim3-4} and arguing similarly to the proof of Proposition \ref{prop3-3} (we retain the same notation), we obtain an exact sequence

%Assume there exists a point $x$ for which $E_2|_{p^{-1}(x)}\cong\mathcal{L}^{\perp}$, and a point $y$ for which $ \mathcal{L}_{ker}\oplus \mathcal{O}_{\mathbb{Q}^{3}}$. Arguing similarly to the proof of Proposition 1 (we retain the same notation), we obtain an exact sequence
 \[
0\rightarrow E \rightarrow T_X(-1) \rightarrow R^1p_*{E_2}\rightarrow 0,
\]
where $R^1p_*{E_2}$ is zero or a torsion sheaf. Since $E$ and $T_X(-1)$ share identical Chern classes, $R^1p_*{E_2}$ vanishes and consequently $E\cong T_X(-1)$.
	\end{proof}
	
	Next, we classify uniform $(2n-1)$-bundles with splitting type $(0,\ldots,0,-1)$ on quadrics $\mathbb{Q}^{2n-1}$ for arbitrary $n\ge 3$. We first prove a key lemma. Here 
	$X_1$ denotes the ample generator of $\mathrm{Pic}(\mathbb{Q}^{2n-1})$
	and we refer the reader to \cite[Lemma 2.1]{FLL} for the Chow rings of $X=\mathbb{Q}^{2n-1}$, $\mathcal{U}=B_n/P_{1,2}$ and $\mathcal{M}=B_n/P_{2}$.
	\begin{lemma}\label{lemma3-5}
		Let $n\ge 2$ and let $E$ be a uniform $(2n-1)$-bundle on $\mathbb{Q}^{2n-1}$ of splitting type $(0,\ldots,0,-1)$. Suppose 
		$c_1(E)=-X_1$ and $c_2(E)=X_1^2$.
		Then $E$ is isomorphic to an orthogonal kernel bundle $\mathcal{L}_{\mathrm{ker}}$.
	\end{lemma}
	\begin{proof}
We prove the proposition by induction over $n$. For $n=2$, the assertion holds by Theorem \ref{uniform Q3}. Suppose the assertion is true for $n-1$. Let $E$ be a vector bundle on $X=\mathbb{Q}^{2n-1}$ satisfying the conditions in the proposition. Then the relative H-N filtration of $p^*E$ induces the following exact sequence
$0\rightarrow  q^*G\rightarrow p^*E \rightarrow q^*\oh_{\mathcal{M}}(a)\otimes p^*\oh_{X}(-1) \rightarrow 0$, where $G$ is a rank $2n-2$ bundle. Since $c_1(E)=-X_1$ and $c_1(q^*\oh_{\mathcal{M}}(a)\otimes p^*\oh_{X}(-1))=a(X_1+X_2)-X_1$, we have $c_1(q^*G)=-a(X_1+X_2)$. Assume $c_2(q^*G)=b_1(X_1^2+X_2^2)+b_2X_1X_2$. Then 
\begin{align*}
c_2(p^*E)&=c_2(q^*G)+c_1(q^*\oh_{\mathcal{M}}(a)\otimes p^*\oh_{X}(-1))c_1(q^*G)\\
&=(b_1+a-a^2) X_1^2+(b_1-a^2)X_2^2+(b_2+a-2a^2)X_1X_2.
\end{align*}
Note that for $n\ge 3$, the generators of the defining ideal of the Chow ring of $\mathcal{U}$ have degree at least $4$. Consequently, $c_2(E)=X_1^2$ implies that 
\[
b_1+a-a^2=1,~b_1-a^2=0~\text{and}~b_2+a-2a^2=0,
\]
which in turn yields $b_2=b_1=a=1$. So $c_1(q^*G)=-(X_1+X_2)$ and $c_2(q^*G)=X_1^2+X_2^2+X_1X_2$. On the other hand, for any $p$-fiber $p^{-1}(x)\cong \mathbb{Q}^{2n-3}$, from the relative H-N filtration, we obtain the exact sequence
\[
0\rightarrow  q^*G|_{p^{-1}(x)}\rightarrow \oh^{2n-1}_{p^{-1}(x)} \rightarrow \oh_{p^{-1}(x)}(1) \rightarrow 0.
\]
This means that $G|_{qp^{-1}(x)}\cong q^*G|_{p^{-1}(x)}$ is a uniform $(2n-2)$-bundle on $\mathbb{Q}^{2n-3}$ with $c_1(G|_{qp^{-1}(x)})=-X_2$ and $c_2(G|_{qp^{-1}(x)})=X_2^2$, whose splitting type is exactly $(0,\ldots,0,-1)$. Since the dual of $G|_{qp^{-1}(x)}$ is globally generated and its rank is greater than $2n-3$, it contains a trivial subbundle of rank $1$. %(see \cite[Corollary 5.5]{EH}). 
So $\oh_{p^{-1}(x)}$ is a quotient bundle of $G|_{qp^{-1}(x)}$. 
 Let $K_x:=\ker(G|_{qp^{-1}(x)}\to \oh_{p^{-1}(x)})$. Then $K_x$ is precisely a uniform bundle of rank $2n-3$ on $\mathbb{Q}^{2n-3}$ with splitting type$(0,\ldots,0,-1)$. In particular, its Chern classes coincide with those of $G|_{qp^{-1}(x)}$.
 By the induction hypothesis, $K_x$ is isomorphic to an orthogonal kernel bundle $\mathcal{L}_{\mathrm{ker}}$ on $\mathbb{Q}^{2n-3}$ for any $x\in X$. By Proposition \ref{prop2}, $h^1(\mathbb{Q}^{2n-3},\mathcal{L}_{\mathrm{ker}})=1$, so $q^*G|_{p^{-1}(x)}$ is isomorphic to either $\mathcal{L}^{\perp}$ or 
 $\mathcal{L}_{\mathrm{ker}}\oplus \mathcal{O}_{\mathbb{Q}^{2n-3}}$. %In fact, from the above analysis and via explicit computations, we can further deduce that $c_i(E)=(-1)^iX_1^i$ for all $1\le i\le 2n-1$.\\
 
 Arguing as in the proof of Claim \ref{claim:1}, with the necessary vector bundles replaced, we conclude that there exist two different points $x$ and $y$ such that $q^*G|_{p^{-1}(x)}\cong \mathcal{L}^{\perp}$ and $q^*G|_{p^{-1}(y)}\cong \mathcal{L}_{\mathrm{ker}}\oplus \mathcal{O}_{\mathbb{Q}^{2n-3}}$.
 Let $S_E:=\{y\in X\mid q^*G|_{p^{-1}(y)}\cong\mathcal{L}_{\mathrm{ker}}\oplus \mathcal{O}_{\mathbb{Q}^{2n-3}}\}$. Then the set $S_E$ is closed in $X$ by the semicontinuity theorem. \\
 
 \begin{claim}\phantomsection\label{claim:one} $S_E$ is the support of the sheaf $R^1p_{*}q^*G$.
 \end{claim}
 
Proof. By Propositions \ref{prop1} and \ref{prop2}, every higher direct image $R^ip_{*}q^*G~(i\ge 2)$ vanishes. Since $R^1p_{*}q^*G$ is zero outside of $S_E$, we have $\mathrm{supp}~ R^1p_{*}q^*G\subset S_E$. \cite[Theorem 12.11]{Har1} further implies the natural map $\phi^1(y)\colon R^1p_{*}q^*G(x)\rightarrow H^1(p^{-1}(x),q^*G|_{p^{-1}(x)})$ is surjective for any point $x\in X$, and is therefore an isomorphism. This yields $\mathrm{supp}~ R^1p_{*}q^*G= S_E$. \\

With the help of the sheaf $F:=R^1p_{*}q^*G$, we provide $S_E$ with a complex structure. By projecting the relative H-N filtration of $p^*E$ 
\[0\rightarrow  q^*G\rightarrow p^*E \rightarrow q^*\oh_{\mathcal{M}}(1)\otimes p^*\oh_{X}(-1) \rightarrow 0
 \]
 onto $X$, we obtain the exact sequence 
  \begin{align}\label{exact3-3}
  	0\rightarrow E \rightarrow T_X(-1) \rightarrow F\rightarrow 0,
  \end{align}
 for $p_{*}q^*G$ is torsion free and zero outside of $S_E$. Denote by $h\colon E \rightarrow T_X(-1)$ the morphism arising from the exact sequence (\ref{exact3-3}). $h$ is then a sheaf monomorphism between locally free sheaves of the same rank with cokernel $F$. 
 
 Let $\mathrm{Im(det}~ h)\subset \det T_X(-1)$ be the image of the determinant $\det h\colon\det E\hookrightarrow\det T_X(-1)$. Then 
 \[
 I_h=\mathrm{Im(det}~ h)\otimes \det \Omega_X(1)\subset \mathcal{O}_X
 \]
 is an invertible sheaf of ideals with $\mathrm{supp}~ \mathcal{O}_X/ I_h=\mathrm{supp}~ R^1p_{*}q^*G= S_E$. We now define the divisor $D_E$ by $D_E=(S_E, \mathcal{O}_X/ I_h)$. Then $[D_E]=\det T_X(-1)-\det E=X_1$. Thus $D_E$ is a prime divisor on $X$ of degree $1$ and the set $S_E$ can be regarded in a canonical fashion as the support of $D_E$. \\
 
 Since $F$ is the cokernel of the morphism $h$, one readily sees that $I_{D_E}F=0$, which implies that $F\cong i_{*}i^{*}F$, where $i\colon D_E\hookrightarrow X$ denotes the closed embedding of $D_E$ into $X$. By the proof of Claim \ref{claim:one}, we have $\dim i^{*}F(y)=h^1(p^{-1}(y),q^*G|_{p^{-1}(y)})=1$ for every $y\in D_E$, so $i^{*}F$ is a rank $1$ coherent sheaf on $D_E$. By Nakayama's lemma, $(i^{*}F)_y$ is cyclic. Moreover,
pulling back the exact sequence (\ref{exact3-3}) along $i$, we obtain a long exact sequence
 \[
 0\rightarrow \mathcal{T}or_1^{\mathcal{O}_X}( i_{*}\mathcal{O}_{D_E},F) \rightarrow i^*E\rightarrow i^*T_X(-1) \rightarrow i^{*}F\rightarrow 0.
 \]
 As $F\cong i_{*}i^{*}F$, we have $\mathcal{T}or_1^{\mathcal{O}_X}( i_{*}\mathcal{O}_{D_E},F)\cong i^{*}F\otimes i^{*}I_{D_E}$. Note that $i^{*}I_{D_E}$ is exactly the conormal bundle $N^{\vee}_{D_E/X}$, so $\mathcal{T}or_1^{\mathcal{O}_X}( i_{*}\mathcal{O}_{D_E},F)\cong i^{*}F\otimes \mathcal{O}_{D_E}(-1)$, where we use that $D_E$ is a degree $1$ prime divisor on $X$ with $\mathrm{Pic}(D_E)\cong \mathbb{Z}$. Since $i^*E$ is locally free, $i^{*}F\otimes \mathcal{O}_{D_E}(-1)$ is a rank $1$ torsion free sheaf on $D_E$, so is $i^{*}F$. Together with the cyclicity of $(i^{*}F)_y$, we obtain that each $(i^{*}F)_y$ is free of rank one, so $i^{*}F$ is a line bundle on $D_E$. \\
 
 Set $i^{*}F=\mathcal{O}_{D_E}(l)$. Then 
 $c_2(F)=c_2(i_{*}\mathcal{O}_{D_E}(l))=-(l-1)[D_E]^2=-(l-1)X_1^2$.
 Meanwhile, from the exact sequence (\ref{exact3-3}) we get $c_2(F)=c_2(T_X(-1))-c_2(E)-c_1(E)c_1(F)=X_1^2$. This forces $l = 0$. Therefore $F$ fits the following exact sequence on $X$:
 \[
 0\rightarrow \mathcal{O}_X(-1) \rightarrow  \mathcal{O}_X\rightarrow F\rightarrow 0.
 \]
 Applying the functor $\operatorname{Hom}(-,\oh_X(-1))$ to the exact sequence $(\ref{exact3-3})$ and adapting the argument from Theorem \ref{uniform Q3}, we deduce that $E$ is isomorphic to an orthogonal kernel bundle. 
\end{proof}

\begin{proposition}\label{prop3-6}
Assume $n$ is at least $3$. Then every uniform $(2n-1)$-bundle on $\mathbb{Q}^{2n-1}$ of splitting type $(0,\ldots,0,-1)$ either splits or is isomorphic to an orthogonal kernel bundle $\mathcal{L}_{\mathrm{ker}}$.
\end{proposition}
\begin{proof}
Let $E$ be a uniform bundle of splitting type $(0,\ldots,0,-1)$ on $\mathbb{Q}^{2n-1}~(n\ge 3)$. Then $E|_{\mathbb{P}^{n-1}}$ is also uniform, for $\mathbb{P}^{n-1}$ is a largest linear space on $\mathbb{Q}^{2n-1}$. By \cite[Proposition 2.2, Page 30]{ellia}, $E|_{\mathbb{P}^{n-1}}$ either splits or is isomorphic to $\Omega_{\mathbb{P}^{n-1}}(1)\oplus \oh^{\oplus n}_{\mathbb{P}^{n-1}}$.\\

 Note that every $2$-plane in $\mathbb{Q}^{2n-1}$ is contained in some $\mathbb{P}^{n-1}$.
If $E|_{\mathbb{P}^{n-1}}$ splits, then by \cite[Corollary 3.6]{DFG1}, $E$ splits. Hence we may assume $E|_{\mathbb{P}^{n-1}}$ is isomorphic to $\Omega_{\mathbb{P}^{n-1}}(1)\oplus \oh^{\oplus n}_{\mathbb{P}^{n-1}}$. Note that the class of $\mathbb{P}^{n-1}$ in the Chow ring of $\mathbb{Q}^{2n-1}$ is $\frac{X_1^n}{2}$ and for any $1\le i\le n-1$, $c_i(\Omega_{\mathbb{P}^{n-1}}(1))=(-1)^i$. Consequently, we have for any $1\le i\le n-1$, $c_i(E)=(-1)^iX^i_1$. By Lemma \ref{lemma3-5}, we are done. 
\end{proof}

For the remaining three splitting types, we shall prove that all uniform rank $5$ vector bundles of these splitting types split.
\begin{proposition}\label{prop3-7}
	Every uniform bundle of rank $5$ on $\mathbb{Q}^5$ with splitting type $(2,1,0,0,0)$ splits.%$(2,1,0,0,0),~(2,2,1,0,0)$ or $(2,1,0,0,-1)$ splits.
\end{proposition}
\begin{proof}
	%\noindent\textbf{Type 1: splitting type $(2,1,0,0,0)$.}
	Let $E$ be a uniform rank $5$ bundle on $\mathbb{Q}^5$ of this splitting type.
	The relative H-N filtration of $p^*E$ yields exact sequences
	\begin{align*}
		&0\rightarrow E_1(=q^*G_1\otimes p^*\oh_{X}(2))\rightarrow  E_2\rightarrow E_2/E_1(=q^*{G_2}\otimes p^*\oh_{X}(1))\rightarrow 0,\\
		&0\rightarrow E_2\rightarrow  p^*E\rightarrow E_3(=q^*{G_3})\rightarrow 0, 
	\end{align*}
	where $G_1$, $G_2$ are line bundles and $G_3$ is a rank $3$ bundle. Assume that $c_1(E_1)=a_1(X_1+X_2)+2X_1$ and $c_1(E_2/E_1)=a_2(X_1+X_2)+X_1$. Then $c_1(E_3)=-(a_1+a_2)(X_1+X_2)$, as $c_1(p^*E)=3X_1$. We further assume $c_2(E_3)=b_1(X_1^2+X_2^2)+b_2X_1X_2$ and $c_3(E_3)=c_1(X_1^3+X_2^3)+c_2X_1X_2(X_1+X_2)$. By Whitney's formula, we obtain the Chern polynomials of $p^*E$ as polynomials in $X_1,X_2$. Since $c_t(p^*E)\equiv \mu_t X_1^t \pmod{I}$ for $1\le t\le 5$, we conclude that the coefficients of $X_2^i~(i=1,2,3)$ and $X_1 X_2^j, ~X_1^jX_2~(j=1,2,3)$ vanish in $c_t(p^*E)$. This gives a system of polynomial equations in the unknowns $a_1,a_2,b_1,b_2,c_1,c_2$.
	Solving this system yields
	$(a_1,a_2)=(0,0)$. Thus $E_2|_{p^{-1}(x)}$ and $E_3|_{p^{-1}(x)}$ are trivial for every $x$, so $E$ splits. 

\end{proof}

\begin{proposition}\label{prop3-8}
	Every uniform bundle of rank $5$ on $\mathbb{Q}^5$ with splitting type $(2,1,0,0,-1)$ splits.%$(2,1,0,0,0),~(2,2,1,0,0)$ or $(2,1,0,0,-1)$ splits.
\end{proposition}
\begin{proof}
Let $E$ be a vector bundle as in the proposition. The relative H‑N filtration of $p^*E$ now consists of three successive short exact sequences
\begin{align*}
	&0\rightarrow E_1\bigl(=q^*G_1\otimes p^*\mathcal{O}_{X}(2)\bigr)\rightarrow  E_2\rightarrow E_2/E_1\bigl(=q^*G_2\otimes p^*\mathcal{O}_{X}(1)\bigr)\rightarrow 0,\\
	&0\rightarrow E_2\rightarrow  E_3\rightarrow E_3/E_2\bigl(=q^*G_3\bigr)\rightarrow 0,\\
	&0\rightarrow E_3\rightarrow  p^*E\rightarrow E_4\bigl(=q^*G_4\otimes p^*\mathcal{O}_{X}(-1)\bigr)\rightarrow 0,
\end{align*}
where $G_1,G_2,G_4$ are line bundles and $G_3$ is a rank $2$ bundle. Assume $c_1(q^*G_i)=a_i(X_1+X_2)$ for $1\le i\le 3$. Then $c_1(q^*{G_4})=-(a_1+a_2+a_3)(X_1+X_2)$. We further assume $c_2(q^*{G_3})=b_1(X_1^2+X_2^2)+b_2X_1X_2$. By computation, one finds that $a_1=a_3=0$ and $a_2=0$ or $-1$. In either case, $E_1=p^*\mathcal{O}_{X}(2)$. Clearly, $E_1$ is a subbundle of $p^*E$. Let $K$ denote the quotient bundle. From $0\rightarrow p^*\mathcal{O}_{X}(2)\rightarrow  p^*E\rightarrow K\rightarrow 0$, we obtain an exact sequence $0\rightarrow \mathcal{O}_{X}(2)\rightarrow  E\rightarrow  p_*K\rightarrow 0$. It is not hard to see that $p_*K$ is a uniform $4$-bundle on $\mathbb{Q}^5$ of splitting type $(1,0,0,-1)$. By \cite[Theorem 4.3]{FLL}, $p_*K$ splits. Hence $E$ also splits.
\end{proof}

We now consider the most subtle type of splitting. Before proceeding to the proof, we first study the geometry of the universal family $\mathcal{U}$ in detail. Denote by $V := T_X(-1)$, and let $\bar{p}\colon \mathbb{P}(V)\to X$ be the associated projective bundle. We then have a commutative diagram
\[
\begin{tikzcd}
	\mathcal{U} \arrow[hookrightarrow, r] \arrow[d, "p" swap] & \mathbb P(V) \arrow[ld, "\bar{p}" swap] \\
	X.
\end{tikzcd}
\]
Recall that $X = \mathbb{Q}^5\subset \mathbb{P}^6$ is defined by a non‑degenerate symmetric bilinear form $\mathcal{Q}$. The form $\mathcal{Q}$ induces a non‑degenerate quadratic form $\mathfrak{q}$ on $V$, that is, $\mathfrak{q}\in  H^0(\mathbb{P}(V),\mathcal{O}_{\mathbb{P}(V)}(2))\cong H^0(X,\operatorname{Sym}^2 V^{\vee})$. On each fiber of $\mathbb{P}(V)$, $\mathfrak{q}_x$ defines a quadric hypersurface which is exactly the fiber $p^{-1}(x)$. Consequently, $\mathcal{U}$ is the zero locus of the relative quadratic form $\mathfrak{q}$. %and its divisor class in $\mathbb{P}(V)$ is $2X_2$. 
Let $T_{\mathcal{U} /X}$ and $T_{\mathbb{P}(V)/X}$ denote the relative tangent bundles. We obtain the exact sequence 
\begin{align}\label{split2}
	0\rightarrow T_{\mathcal{U} /X}\rightarrow  T_{\mathbb{P}(V)/X}|_\mathcal{U}\rightarrow \oh_\mathcal{U} (2)\rightarrow 0,
\end{align}
where $\oh_\mathcal{U} (2)$ is the line bundle on $\mathcal{U}$ with $c_1(\oh_\mathcal{U} (2))=2X_2$.

\begin{proposition}\label{prop3-9}
	Every uniform bundle of rank $5$ on $\mathbb{Q}^5$ with splitting type $(2,2,1,0,0)$ splits.
\end{proposition}
\begin{proof}
Let $E$ be a vector bundle on $X=\mathbb{Q}^5$ satisfying the conditions. The relative H‑N filtration of $p^*E$ induces
\begin{align}
	&0\rightarrow E_1(=q^*G_1\otimes p^*\oh_{X}(2))\rightarrow  E_2\rightarrow E_2/E_1(=q^*{G_2}\otimes p^*\oh_{X}(1))\rightarrow 0,\label{exA}\\
	&0\rightarrow E_2\rightarrow  p^*E\rightarrow E_3(=q^*{G_3})\rightarrow 0, \label{exB}
\end{align}
where $G_1,G_3$ are rank $2$ bundles and $G_2$ is a line bundle.
	Write $c_1(q^*G_1)=a_1(X_1+X_2)$, $c_1(q^*G_2)=a_2(X_1+X_2)$, so $c_1(E_3)=-(a_1+a_2)(X_1+X_2)$, and set
	$c_2(q^*G_1)=b_1(X_1^2+X_2^2)+b_2X_1X_2$, $c_2(E_3)=b_1'(X_1^2+X_2^2)+b_2'X_1X_2$. A direct computation using the condition $c_t(p^*E)\equiv\mu_tX_1^t \pmod{I}$ yields
	$a_1=a_2=0$ or $a_1=-2,~ a_2=0$. In the former case, $E$ splits. We next rule out the latter case. Suppose the latter holds, i.e., $a_1=-2,~ a_2=0$. Then $E_2/E_1=p^*\oh_{X}(1)$ and one further computes that $b_1=b_2=b_1'=b_2'=2$ and hence $c_1(E)=5X_1$, $c_2(E)=12X_1^2$. Moreover,
	for every fiber $p^{-1}(x)(\cong \mathbb{Q}^3)$, $E_3|_{p^{-1}(x)}$ is a rank $2$ bundle on $\mathbb{Q}^3$ with $c_1(E_3|_{p^{-1}(x)})=2X_2$, $c_2(E_3|_{p^{-1}(x)})=2X_2^2$ and it is globally generated. By \cite[Theorem 1.1]{BHM}, $E_3|_{p^{-1}(x)}\cong {l_x}^* N(1)$, where $N$ is a null correlation bundle on $\mathbb{P}^3$ and 
	${l_x}$ is the restriction of the linear projection $\mathbb{P}^4\setminus \{[z_x]\}\rightarrow \mathbb{P}^3$ to $p^{-1}(x)$ with $[z_x]\notin p^{-1}(x)$. Similarly, since $E_1^{\vee}|_{p^{-1}(x)}$ is a bundle on each $p$-fiber satisfying the same conditions, we have $E_1^{\vee}|_{p^{-1}(x)}\cong {l_x}^* N'(1)$ for some null correlation bundle $N'$, and hence $E_1|_{p^{-1}(x)} \cong  {l_x}^* N'(-1)$.
	
	Since $(l_x)_{*}\oh_{\mathbb{Q}^3}=\oh_{\mathbb{P}^3}\oplus \oh_{\mathbb{P}^3}(-1)$, the projection formula yields $(l_x)_*{l_x}^* N(1)\cong N(1)\otimes(l_x)_*\oh_{\mathbb{Q}^3}=N(1)\oplus N$. Given that $H^0(\mathbb{P}^3, N)=0$, we have
\[
H^0(p^{-1}(x), E_3|_{p^{-1}(x)})\cong H^0(\mathbb{P}^3, (l_x)_*{l_x}^* N(1))\cong H^0(\mathbb{P}^3, N(1))\oplus H^0(\mathbb{P}^3, N)=H^0(\mathbb{P}^3, N(1)).
\]
Since $h^0(\mathbb{P}^3, N(1))=5$, $p_{*}E_3$ is a vector bundle of rank $5$. Furthermore, for all $i>0$, the isomorphic groups $H^i(p^{-1}(x), E_3|_{p^{-1}(x)})\cong H^i(\mathbb{P}^3, (l_x)_*{l_x}^* N(1))$ vanish. Consequently, $R^ip_{*}E_3=0$ for all $i>0$. Restricting (\ref{exA}) to $p^{-1}(x)$, we get an exact sequence $\rightarrow {l_x}^* N'(-1)\rightarrow  E_2|_{p^{-1}(x)}\rightarrow \oh_{\mathbb{Q}^3}\rightarrow 0$. Since $\dim \mathrm{Ext}^1(\oh_{\mathbb{Q}^3},{l_x}^* N'(-1))=1$ and ${l_x}^*T_{\mathbb{P}^3}(-2)$ is a non-trivial extension of $\oh_{\mathbb{Q}^3}$ by ${l_x}^* N'(-1)$, it follows that
$E_2|_{p^{-1}(x)}$ is isomorphic to $ {l_x}^* N'(-1)\oplus \oh_{\mathbb{Q}^3}$ or ${l_x}^*T_{\mathbb{P}^3}(-2)$. 
We now rule out the case $a_1=-2,~ a_2=0$ by examining the following three cases, each of which will lead to a contradiction.\\

\textbf{Case 1.} There exists a point $x$ such that $E_2|_{p^{-1}(x)}\cong {l_x}^*T_{\mathbb{P}^3}(-2)$ and a point $y$ such that $E_2|_{p^{-1}(y)}\cong {l_y}^* N'(-1)\oplus \oh_{\mathbb{Q}^3}$.

Let $S_E:=\{y\in X\mid E_2|_{p^{-1}(y)}\cong{l_y}^* N'(-1)\oplus \oh_{\mathbb{Q}^3}\}$. Then the set $S_E$ is a non‑empty closed subset of $X$. As in Claim \ref{claim:one}, $S_E$ coincides with the support of $R^1p_*{E_2}$. Since $p_*{E_2}$ is torsion free and zero outside of $S_E$, we conclude that $p_*{E_2}$ vanishes, which yields the following short exact sequence:
\[
0\rightarrow E \rightarrow p_{*}E_3 \rightarrow R^1p_*{E_2}\rightarrow 0,
\]
where $E$ and $p_{*}E_3$ are both vector bundles of rank $5$. The monomorphism between them induces a monomorphism $\det E\hookrightarrow \det p_{*}E_3$. Next, we calculate the first Chern class of $p_{*}E_3$. Since $c_1(E_3)=2(X_1+X_2)$, $c_2(E_3)=2(X_1^2+X_2^2+X_1X_2)$, we have the Chern character
\[
ch(E_3)=2+2(X_1+X_2)+2X_1X_2+(-\frac{2}{3})(X_1^3+X_2^3)+(-\frac{2}{3})X_1X_2(X_1^2+X_2^2)+\cdots.
\]
Moreover, combining the relative Euler sequence $0\rightarrow\oh_{\mathbb{P}(V)} \rightarrow \bar{p}^*V\otimes\oh_{\mathbb{P}(V)}(1)\rightarrow T_{\mathbb{P}(V)/X}\rightarrow 0$ with (\ref{split2}), a direct computation gives the Todd class
\[
td(T_{\mathcal{U} /X})=1+\frac{3}{2}X_2+\frac{1}{12}(13X_2^2+X_1^2)+\frac{1}{24}(12X_2^3+3X_1^2X_2)+\frac{1}{720}(117X_2^4+3X_1^4+63X_1^2X_2^2)+\cdots.
\] 
 From the preceding analysis, we have $R^ip_{*}E_3=0$ for all $i>0$. Therefore, by Grothendieck-Riemann-Roch theorem, the Chern character $ch(p_{*}E_3)=p_*(ch(E_3).td(T_{\mathcal{U} /X}))$ in the Chow ring of $X$. Substituting the explicit forms of $ch(E_3)$ and $td(T_{\mathcal{U} /X})$ into this identity and applying the Segre class formula, one computes $c_1(p_{*}E_3)=5X_1=c_1(E)$. Therefore, the monomorphism $\det E\hookrightarrow \det p_{*}E_3$ is an isomorphism. This forces $E\cong p_{*}E_3$. Hence $R^1p_*{E_2}=0$. This contradicts the non‑emptiness of the support of $R^1p_*{E_2}$.\\

\textbf{Case 2.} For any $x\in X$, $E_2|_{p^{-1}(x)}$ is isomorphic to  ${l_x}^* N'(-1)\oplus \oh_{\mathbb{Q}^3}$.

Since for each $x\in X$ we have $h^0(\mathbb{Q}^{3},{l_x}^* N'(-1)\oplus \oh_{\mathbb{Q}^3})=h^1(\mathbb{Q}^{3},{l_x}^* N'(-1)\oplus \oh_{\mathbb{Q}^3})=1$, it follows from the base change theorem that $p_{*}E_2$ and $R^1p_{*}E_2$ are line bundles. Suppose that $p_{*}E_2=\oh_X(a)$ and $R^1p_{*}E_2=\oh_X(b)$. By performing $R^ip_*$ to (\ref{exB}), we have the exact sequence $0\rightarrow \oh_X(a) \rightarrow E \rightarrow p_{*}E_3\rightarrow \oh_X(b)\rightarrow 0 $. The previous argument gives $c_1(p_{*}E_3)=c_1(E)$, hence $a=b$. Applying $p_*$ to (\ref{exA}) yields the long exact sequence (noting that $p_*E_1=0$)
 \[
 0\rightarrow p_*E_2(=\oh_X(a)) \rightarrow p_*p^*\oh_{X}(1) \rightarrow R^1p_*E_1\rightarrow R^1p_{*}E_2(=\oh_X(b))\rightarrow 0.
 \]
 Since $a=b$, $R^1p_*E_1$ must be the line bundle $\oh_X(1)$. This forces $a=1$ and the induced morphism $p_*f\colon p_*E_2 \rightarrow p_*p^*\oh_{X}(1)$ is an isomorphism, where $f\colon E_2 \rightarrow p^*\oh_{X}(1)$ denotes the surjection in (\ref{exA}). Let $g\colon p^*p_*E_2\rightarrow E_2$ be the natural morphism, and consider its composition with $f$. By naturality of the adjunction counit, $f\circ g=\varepsilon_{p^*\mathcal O_X(1)}\circ p^*p_*f$, where $\varepsilon_{p^*\mathcal O_X(1)}\colon p^*p_*\big(p^*\mathcal O_X(1)\big)\to p^*\mathcal O_X(1)$ is the counit morphism. The projection formula ensures that $\varepsilon_{p^*\mathcal O_X(1)}$ 
 is an isomorphism. As \(p_*f\) is an isomorphism, so is \(p^*p_*f\). Hence \(f\circ g\) is an automorphism of the line bundle $p^*p_*E_2\cong p^*\mathcal O_X(1)$. Consequently, $f\circ g = c\cdot\operatorname{id}_{p^*\mathcal O_X(1)}$, $c\in \mathbb{C} \setminus \{0\}$. 
 Thus $f$ is a split surjection, which forces $E_2\cong E_1\oplus p^*\oh_{X}(1)$. 
 
Clearly, $p^*\oh_{X}(1)$ is a subbundle of $p^*E$. Let $K$ denote the quotient bundle. Then we have $K\cong p^*p_*K$ and 
the exact sequence $0\rightarrow \mathcal{O}_{X}(1)\rightarrow  E\rightarrow  p_*K\rightarrow 0$. Moreover, since $E_2/p^*\oh_{X}(1)\cong E_1$, the snake lemma yields an exact sequence $0\to E_1\to K\to E_3\to 0$, whose restriction to $q$-fibers looks as: $0\to \oh_{q^{-1}(l)}(2)^{\oplus 2}\to K|_{q^{-1}(l)}\to \oh_{q^{-1}(l)}^{\oplus 2}\to 0$. Thus over every line $q^{-1}(l)$, $K|_{q^{-1}(l)}\cong \oh_{q^{-1}(l)}^{\oplus 2}(2)\oplus \oh_{q^{-1}(l)}^{\oplus 2}$. Consequently, $p_*K$ is a uniform $4$-bundle on $X$ of splitting type $(2,2,0,0)$. By \cite[Theorem 4.3]{FLL}, $p_*K$ splits. Hence $E$ also splits, which contradicts $c_2(E)=12X_1^2$. We therefore rule out this case.\\

\textbf{Case 3.} For any $x\in X$, $E_2|_{p^{-1}(x)}$ is isomorphic to  ${l_x}^* T_{\mathbb{P}^3}(-2)$.

Since $h^0(\mathbb{Q}^{3},{l_x}^* T_{\mathbb{P}^3}(-2))=h^1(\mathbb{Q}^{3},{l_x}^* T_{\mathbb{P}^3}(-2))=0$, %(see Proposition \ref{prop1}), 
the direct images $p_{*}E_2$ and $R^1p_{*}E_2$ vanish. Applying $R^ip_*$ to (\ref{exB}), we obtain $E\cong p_{*}E_3$. This induces a morphism $\wedge^{2}E\cong \wedge^{2}(p_{*}E_3)\to p_{*}(\wedge^{2}E_3)=p_{*}(\det E_3)$.
The equality $c_1(E_3)=2(X_1+X_2)$ implies that $\det E_3=q^*\oh_{\mathcal{M}}(2)=p^*\oh_X(2)\otimes \oh_\mathcal{U} (2)$. 
Consequently, by the projection formula, $p_{*}(\det E_3)\cong\oh_X(2)\otimes p_{*}\oh_\mathcal{U} (2)$. Tensoring the above morphism with $\oh_X(-2)$, we get a morphism \[\delta\colon\wedge^{2}E\otimes\oh_X(-2)\to p_{*}\oh_\mathcal{U} (2).\]
Applying $\bar{p}_*$ to the exact sequence $ 0\rightarrow \oh_{\mathbb{P}(V)} \xrightarrow{\mathfrak{q}} \oh_{\mathbb{P}(V)}(2) \rightarrow \oh_\mathcal{U} (2)\rightarrow 0$, we obtain an exact sequence $ 0\rightarrow \oh_X \xrightarrow{\mathfrak{q}} \bar{p}_*\oh_{\mathbb{P}(V)}(2)(\cong \operatorname{Sym}^2 V^{\vee}) \rightarrow p_{*}\oh_\mathcal{U} (2)\rightarrow 0$. It follows that $p_{*}\oh_\mathcal{U} (2)\cong \operatorname{Sym}^2 V^{\vee}/\oh_X$. Let $\langle \mathfrak{q}\rangle$ denote the image of the injection $\mathfrak{q}$. Since
 $\mathfrak{q}_x$ is non‑degenerate for every $x\in X$, we have $\langle\mathfrak{q}\rangle\cong \oh_X$.
 
\textbf{Claim.} $\mathrm{Im}(\delta)\subset \operatorname{Sym}^2 V^{\vee}/\langle \mathfrak{q}\rangle$ is a subbundle of rank $10$. 

Proof. Let $V_x$ denote the fiber of $V=T_X(-1)$. The non-degenerate quadratic form $\mathfrak{q}_x$ on $V_x$ defines the quadric $p^{-1}(x)$. %Then the non‑degenerate quadratic form $\mathfrak{q}_x$ on $V_x$ defines a quadric, which is exactly the fiber $p^{-1}(x)$. 
As established at the beginning of the proof, $E_3|_{p^{-1}(x)} = {l_x}^* N(1)$, where
${l_x}$ is the restriction to $p^{-1}(x)$ of the linear projection $\mathbb{P}(V_x)\setminus \{[z_x]\}\rightarrow \mathbb{P}(V_x/\mathbb{C}z_x)$ and $[z_x]$ is
the center of projection. Since 
$[z_x]\notin p^{-1}(x)$, $\mathfrak{q}_x(z_x)\ne 0$. Therefore $\mathfrak{q}_x|_{\mathbb{C}z_x}$ is non-degenerate, which yields the orthogonal direct sum decomposition
 $V_x=\mathbb{C}z_x\oplus W_x$, where $W_x$ is the $\mathfrak{q}_x$-orthogonal complement of $\mathbb{C}z_x$. So $V_x/\mathbb{C}z_x$ is canonically identified with $W_x$.

For each point $x\in X$, the morphism $\delta$ induces a linear map on fibers $\wedge^{2}E_x\to (p_{*}\oh_\mathcal{U} (2))_x$. 
Since $E\cong p_{*}E_3$, we have $E_x\cong H^0(p^{-1}(x),E_3|_{p^{-1}(x)})\cong H^0(\mathbb{P}(W_x),N(1))$, which is a $5$-dimensional irreducible representation of $Sp(4)$ by the Borel-Bott-Weil theorem. Consequently, $\wedge^{2}E_x$ is a $10$-dimensional irreducible representation. It follows from Schur's lemma that the natural map
\[
\wedge^{2}E_x\cong \wedge^{2} H^0(\mathbb{P}(W_x),N(1)) \to H^0(\mathbb{P}(W_x),\det N(1))=H^0(\mathbb{P}(W_x),\oh_{\mathbb{P}(W_x)}(2))
\]
is an isomorphism, since the target is also a $10$-dimensional irreducible representation and the map is nonzero. Moreover, $(p_{*}\oh_\mathcal{U} (2))_x\cong H^0(p^{-1}(x),\oh_{p^{-1}(x)}(2))\cong H^0(\mathbb{P}(W_x),\oh_{\mathbb{P}(W_x)}(2)\oplus \oh_{\mathbb{P}(W_x)}(1))$. Since $\wedge^{2}E_x\cong H^0(\mathbb{P}(W_x),\oh_{\mathbb{P}(W_x)}(2))$ for every $x\in X$, $\delta$ is a morphism of constant rank $10$ and thus its image $\mathrm{Im}(\delta)\subset \operatorname{Sym}^2 V^{\vee}/\langle \mathfrak{q}\rangle$ is a subbundle of rank $10$.\\

Let $\mathfrak{S}$ be the preimage of $\mathrm{Im}(\delta)$ under the quotient map $\operatorname{Sym}^2 V^{\vee}\to \operatorname{Sym}^2 V^{\vee}/\langle \mathfrak{q}\rangle$.
Then $\mathfrak{S}$ fits into the exact sequence $0\rightarrow \oh_X(=\langle \mathfrak{q}\rangle) \rightarrow \mathfrak{S}\rightarrow \mathrm{Im}(\delta)\rightarrow 0$
 and its fiber over $x$ is $\mathrm{Sym}^2 W_x^{\vee}\oplus \mathbb{C}\mathfrak{q}_x$. Consider the pullback of $\mathfrak{S}$ along $\bar{p}$, followed by the relative polarization. At each point $u=(x,[v])\in \mathbb{P}(V)$, this polarization gives the map
 \[
 \gamma\colon \bar{p}^*\mathfrak{S}\longrightarrow \bar{p}^*{V^\vee}\otimes \mathcal O_{\mathbb{P}(V)}(1),\quad s\longmapsto \bigl(v\mapsto s(v,-)\bigr). 
 \]
 Let $D:=D_1(\gamma)=\{u=(x,[v])\in \mathbb{P}(V)\mid \mathrm{rk}(\gamma(u))\le 1\}$ be the degeneracy locus. We next prove that for a fixed point $x\in X$, $D_x=D\cap \mathbb{P}(V_x)$ is set‑theoretically the single point $(x,[z_x])$. Since $V_x=\mathbb{C}z_x\oplus W_x$, any vector $v\in V_x$ can be written as $v= a z_x + w$ with $a\in\mathbb{C}$ and $w\in W_x$.
 
 (1) Suppose $w\neq 0$. Then there exists $\ell\in W_x^\vee$ such that $\ell(w)\neq 0$. Consider $s=\ell^2\in \mathrm{Sym}^2W_x^\vee$, its associated symmetric bilinear form is given by $s(v,v')=\ell(v)\ell(v')$. Since $\ell(v)=\ell(a z_x + w)=\ell(w)\neq 0$, the linear functional $s(v,-)=\ell(w)\,\ell$ is non-zero. Varying $\ell$ over a basis of $W_x^\vee$, we see that $W_x^\vee\subset \mathrm{Im}(\gamma(u))$, hence $\mathrm{rk}(\gamma(u))\ge 4$. Thus $(x,[v])\notin D_x$.
 
 (2) Now suppose $w=0$, so $v= a z_x$. For every $s\in \mathrm{Sym}^2(W_x^\vee)$, we have $s(z_x,-)=0$, while $\mathfrak{q}_x(z_x,-)\neq 0$. Therefore $\mathrm{rk}(\gamma(u))= 1$.\\
 
 Combining (1) and (2), we obtain $\mathrm{supp}~D_x=\{(x,[z_x])\}$.
A local computation around $(x,[z_x])$ in suitable affine coordinates shows that the ideal sheaf of $D_x$ coincides with the maximal ideal of $(x,[z_x])$. Consequently, $D_x$ is a reduced point of length $1$ for any $x\in X$. Thus $\bar{p}|_{D}\colon D\to X$ is quasi-finite, and since it is proper, it is finite. On each fiber, the natural map $\eta\colon \oh_X\to (\bar{p}|_{D})_*\oh_D$ induces an isomorphism $\mathbb{C}\to \mathbb{C}$. By Nakayama's lemma, $\eta$ is surjective. Since $X$ is smooth and $\ker(\eta)$ is zero at the generic point, it must be zero. Thus $\eta$ is an isomorphism. As finite morphisms are affine, we have $D\cong X$. This gives a section $\sigma\colon X\to \mathbb{P}(V)$, equivalently a line subbundle $L\subset V(=T_X(-1))$. Precisely, $L=\sigma^*\mathcal O_{\mathbb{P}(V)}(-1)$, its fibre over $x\in X$ is the line $\mathbb{C}z_x$. Since $\mathfrak{q}_x|_{\mathbb{C}z_x}$ is non-degenerate, the restriction of $\mathfrak{q}$ gives a nowhere‑vanishing morphism
 \[
 \mathfrak{q}\colon L\otimes L\longrightarrow \mathcal O_X.
 \]
As any nowhere-vanishing homomorphism between line bundles is an isomorphism, we obtain $L\otimes L\cong \mathcal O_X$. Since $\mathrm{Pic}(X)\cong\mathbb Z$ has no $2$-torsion, the isomorphism $L^{\otimes 2}\cong\mathcal O_X$ forces $L\cong\mathcal O_X$. We thus obtain a trivial subbundle of $T_X(-1)$. However, $H^0(X,T_X(-1))=0$ by the Borel-Bott-Weil theorem, a contradiction. 

Combining the above, all three cases lead to contradictions. Hence $E$ must split.

\end{proof}

Combining Theorem \ref{uniform Q3}, Propositions \ref{prop3-2}–\ref{prop3-4}, \ref{prop3-6}–\ref{prop3-9}, we complete the proof of Theorem \ref{main1}. 

	\section{Uniform non-homogeneous bundles on $\mathbb{Q}^n$}
	
For any quadric	$\mathbb{Q}^n$ with $n\ge 3$, Proposition \ref{prop2} shows that every orthogonal kernel bundle $\mathcal{L}_{\mathrm{ker}}$ is uniform of splitting type $(0,\ldots,0,-1)$. In this section we prove that $\mathcal{L}_{\mathrm{ker}}$ is not homogeneous.
To begin with, the tangent bundle $T_{\mathbb{Q}^n}$ has splitting type $(2,1,\ldots,1,0)$. Meanwhile, \cite[Corollary 1.6]{ott1} states that every spinor bundle $\mathcal{S}$ is of splitting types $(1,\ldots,1,0,\ldots,0)$ with exactly $2^{\lfloor  \frac{n-3}{2}\rfloor}$ entries equal to $1$. Consequently, $\mathcal{L}_{\mathrm{ker}}$ cannot be isomorphic to any twist of $T_{\mathbb{Q}^n}$, $\mathcal{S}$ or their dual bundles.
	
	Recall that every homogeneous vector bundle over $G/P$ admits a filtration whose successive quotients are irreducible homogeneous bundles. If all these quotients are line bundles, the bundle splits, as line bundles are ACM. Hence, any unsplit homogeneous bundle must have at least one quotient of rank
	$\ge 2$. Accordingly, the rank of any unsplit homogeneous bundle is bounded below by
the minimal rank among irreducible homogeneous bundles of rank $\ge 2$. With this in mind, we proceed to compute the ranks of irreducible homogeneous bundles over quadrics of odd and even dimension separately. Let us first briefly review Weyl's formula.\\
	%Since every homogeneous vector bundle on $G/P$ admits a filtration by irreducible homogeneous bundles, the rank of any homogeneous bundle is evidently bounded below by the minimal rank among irreducible homogeneous bundles. Accordingly, we next compute the ranks of irreducible homogeneous bundles over odd-dimensional and even-dimensional quadrics separately. Let us first briefly review Weyl's formula. \\
	
	Let $G$ be a simple Lie group and $B$ a Borel subgroup of $G$. A classical formula of H. Weyl expresses the dimension of the irreducible representation $G_{\lambda}$ of $G$ with highest weight $\lambda$. Denote by $\Phi^{+}_G\subset \Phi_G$ the subset of positive roots with respect to some ordering of the root system and by $\Delta=\{\alpha_1,\ldots,\alpha_n\}$ a base for $\Phi^{+}_G$.
	%and $\Phi_G$ its root system. Denote by $\Phi^{+}_G\subset \Phi_G$ the subset of positive roots with respect to some ordering of the root system and by $\Delta$ a base for $\Phi^{+}_G$. %Let $\lambda_{1},\ldots,\lambda_{n}\in \Lambda$ be the {\it fundamental weights\/}, i.e., $\frac{2(\lambda_i,\alpha_j)}{(\alpha_j,\alpha_j)}=\delta_{ij}$ for $\alpha_{j}\in \Delta$, where $(,)$ denotes the Killing form. 
Set $\rho=\sum\limits_{i=1}^{n}\lambda_i$ (sum of all the fundamental weights) and let $(,)$ be the Killing form.
Then Weyl's formula is 
\[
\dim G_{\lambda}=\prod\limits_{\alpha\in \Phi^{+}_G}\frac{(\lambda+\rho,\alpha)}{(\rho,\alpha)}.
\]
	
	%By the Weyl formula, this minimal rank equals $\min\{E_{\lambda_2},E_{\lambda_3},\ldots,E_{\lambda_n}\}$, where each $E_{\lambda_i}$ is the irreducible homogeneous vector bundle with highest weight $\lambda_i$. It therefore suffices to compute the ranks of $E_{\lambda_i}$ to obtain the minimal possible rank of homogeneous bundles.
	\textbf{Odd-dimensional quadrics $\mathbb{Q}^{2n-1}\cong B_n/P_1~(n\ge 2)$:}
	
Recall that the Lie algebra of type $B_n$ is a simple Lie algebra %with the subset of positive roots is
 whose set of positive roots is
\[
\Phi^{+}_{B_n}=\{\alpha_i~(1\le i\le n),~ \alpha_i+\cdots+\alpha_j~(1\le i<j\le n),~ \alpha_i+\cdots+\alpha_{j-1}+2(\alpha_j+\cdots+\alpha_n)~(1\le i<j\le n)\}.
\] 
Let $\lambda_{1},\ldots,\lambda_{n}$ be the fundamental weights of $B_n$. Then 
\[
(\lambda_i,\alpha_j)=0,~\text{and}~(\lambda_i,\alpha_i)=1~\text{for}~1\le i\le n-1, ~(\lambda_n,\alpha_n)=\frac{1}{2}.
\]
Let $E_{\lambda}$ be the irreducible homogeneous vector bundle with highest weight $\lambda=\sum\limits_{i=1}^{n}a_i\lambda_i$ on $\mathbb{Q}^{2n-1}$, where $a_i\ge 0$ for $2\le i\le n$. Denote by $\pi$ the projection $ B_n/B\rightarrow B_n/P_1=\mathbb{Q}^{2n-1}$. According to \cite[Proposition 10.13]{ott}, $E_{\lambda}\cong\pi_{*}L_{\lambda}$, where $L_{\lambda}$ is a line bundle with weight $\lambda$ on $B_n/B$.  
Therefore, $\mathrm{rk} E_{\lambda}=\dim H^0(\pi^{-1}(x),L_{\lambda}|\pi^{-1}(x))$ for every $x\in \mathbb{Q}^{2n-1}$. Note that $\pi^{-1}(x)\cong B_{n-1}/B$. So by the Borel-Weil-Bott theorem, $H^0(\pi^{-1}(x),L_{\lambda}|\pi^{-1}(x))$ is the irreducible representation of $B_{n-1}$ with highest weight $\lambda'=\sum\limits_{k=1}^{n-1}a_{1+k}\lambda_k$. It follows from Weyl's formula that 
\begin{align*}
	\mathrm{rk} E_{\lambda}=&\prod\limits_{\alpha\in \Phi^{+}_{B_{n-1}}}\frac{(\lambda'+\rho,\alpha)}{(\rho,\alpha)}=\prod\limits_{1\le i<j\le n-1}\frac{\sum\limits_{k=i}^{j-1}a_{1+k}+j-i}{j-i}\times
	\prod\limits_{1\le i\le n-1}\frac{2n-1-2i+2\sum\limits_{k=2}^{n-2}a_{1+k}+a_n}{2n-1-2i}
	\times\\
	&\prod\limits_{1\le i<j\le n-1}\frac{2n-1-(i+j)+\sum\limits_{k=i}^{j-1}a_{1+k}+2\sum\limits_{k=j}^{n-2}a_{1+k}+a_n}{2n-1-(i+j)}.
\end{align*}
%In particular, one can readily compute that for each fundamental weight $\lambda_i~(2\le i\le n)$, the ranks of $E_{\lambda_i}$ are respectively\begin{align}\label{formula1}\mathrm{rk} E_{\lambda_i}=\frac{(2n-1)!}{(2n-i)!(i-1)!}~\text{for}~2\le i\le n-1, \quad\mathrm{rk} E_{\lambda_n}=2^{n-1}.\end{align}\\
In particular, a direct computation shows that the rank of \(E_{\lambda_i}\) for a fundamental weight \(\lambda_i\) is
\begin{align}\label{formula1}
	\mathrm{rk} E_{\lambda_i} = \frac{(2n-1)!}{(2n-i)!\,(i-1)!} \quad (2 \le i \le n-1), \qquad
	\mathrm{rk} E_{\lambda_n} = 2^{n-1}. 
\end{align}\\

\textbf{Even-dimensional quadrics $\mathbb{Q}^{2n}\cong D_{n+1}/P_1~(n\ge 3)$:}

Recall that %the Lie algebra of type $D_n$ is a simple Lie algebra with simple root system $\Delta=\{\alpha_1,\ldots,\alpha_n\}$. The subset of positive roots is 
\iffalse
\begin{align*}
\Phi^{+}_{D_{n+1}}=\{&\alpha_i~(1\le i\le n+1),~\alpha_i+\cdots+\alpha_{j-1}+2(\alpha_j+\cdots+\alpha_{n-1})+\alpha_n+\alpha_{n+1}~(1\le i<j\le n-1),\\
&\alpha_i+\cdots+\alpha_j~(1\le i<j\le n+1,i\ne n),~ \alpha_i+\cdots+\alpha_{n-1}+\alpha_{n+1}~(1\le i\le n-1)\}.
\end{align*}
\fi
\begin{align*}
	\Phi^{+}_{D_n}=\{&\alpha_i~(1\le i\le n),~\alpha_i+\cdots+\alpha_{j-1}+2(\alpha_j+\cdots+\alpha_{n-2})+\alpha_{n-1}+\alpha_n~(1\le i<j\le n-2),\\
	&\alpha_i+\cdots+\alpha_j~(1\le i<j\le n,i\ne n-1),~ \alpha_i+\cdots+\alpha_{n-2}+\alpha_n~(1\le i\le n-2)\}.
\end{align*}
Let $\lambda_{1},\ldots,\lambda_{n}$ be the fundamental weights of $D_n$. Then 
\[
(\lambda_i,\alpha_j)=0,~\text{and}~(\lambda_i,\alpha_i)=1~\text{for}~1\le i\le n.
\]
Let $E_{\lambda}$ be the irreducible homogeneous vector bundle with highest weight $\lambda=\sum\limits_{i=1}^{n+1}a_i\lambda_i$ on $\mathbb{Q}^{2n}$, where $a_i\ge 0$ for $2\le i\le n+1$. Denote $\lambda'=\sum\limits_{k=1}^{n}a_{1+k}\lambda_k$. 
Similar to the odd‑dimensional case, by Weyl's formula, we have
\begin{align*}
	\mathrm{rk} E_{\lambda}=&\prod\limits_{\alpha\in \Phi^{+}_{D_n}}\frac{(\lambda'+\rho,\alpha)}{(\rho,\alpha)}\\
	&=\prod\limits_{1\le i<j\le n-2}\frac{\sum\limits_{k=i}^{j-1}a_{1+k}+j-i}{j-i}\times
	\prod\limits_{1\le i<j\le n-2}\frac{2n-(i+j)+\sum\limits_{k=i}^{j-1}a_{1+k}+2\sum\limits_{k=j}^{n-2}a_{1+k}+a_n+a_{n+1}}{2n-(i+j)}
	\times\\
	&\prod\limits_{1\le i\le n-2}\frac{n-1-i+\sum\limits_{k=i}^{n-2}a_{1+k}}{n-1-i}\times\prod\limits_{1\le i\le n-2}\frac{n+1-i+\sum\limits_{k=i}^{n}a_{1+k}}{n+1-i}	\times\\
	&\prod\limits_{1\le i\le n-2}\frac{n-i+\sum\limits_{k=i}^{n-1}a_{1+k}}{n-i}\times\prod\limits_{1\le i\le n-2}\frac{n-i+\sum\limits_{k=i}^{n}a_{1+k}}{n-i}.
\end{align*}
%In particular, one computes for each fundamental weight $\lambda_i~(2\le i\le n+1)$, the ranks of $E_{\lambda_i}$ are respectively\begin{align}\label{formula2}\mathrm{rk} E_{\lambda_i}=\frac{(2n)!}{(2n+1-i)!(i-1)!}~\text{for}~2\le i\le n-1, \quad\mathrm{rk} E_{\lambda_n}=\mathrm{rk} E_{\lambda_{n+1}}=2^{n-1}.\end{align}
In particular, for each fundamental weight \(\lambda_i\) with \(2 \le i \le n+1\), the rank of \(E_{\lambda_i}\) is given by
\begin{align}\label{formula2}
	\mathrm{rk} E_{\lambda_i} = \frac{(2n)!}{(2n+1-i)!\,(i-1)!} \quad (2 \le i \le n-1), \qquad
	\mathrm{rk} E_{\lambda_n} = \mathrm{rk} E_{\lambda_{n+1}} = 2^{n-1}. 
\end{align}
	\begin{theorem}\label{non-homo}
		Every orthogonal kernel bundle $\mathcal{L}_{\mathrm{ker}}$ on $\mathbb{Q}^{n}~(n\ge 3)$ is a uniform non-homogeneous bundle of rank $n$. 
	\end{theorem}
	\begin{proof}
By Weyl's formula, any irreducible homogeneous bundle $E_{\lambda}$ on $\mathbb{Q}^{2n-1}$ and $\mathbb{Q}^{2n}$ of rank at least $2$ satisfies
\[
\mathrm{rk}(E_{\lambda})\ge \min\{\mathrm{rk}(E_{\lambda_2}), \ldots,\mathrm{rk}(E_{\lambda_{n-1}}), \mathrm{rk}(E_{\lambda_n})\}.
\]
By Formulas \ref{formula1} and \ref{formula2}, $E_{\lambda_2}$ is the bundle of minimal rank among $\{E_{\lambda_i}\}_{2\le i\le n-1}$, so we further obtain $\mathrm{rk}(E_{\lambda})\ge \min\{\mathrm{rk}(E_{\lambda_2}), \mathrm{rk}(E_{\lambda_n})\}$.
For the quadric $\mathbb{Q}^{2n-1}$ with $n\ge 4$, we have $\mathrm{rk}(E_{\lambda_2})=2n-1<2^{n-1}=\mathrm{rk}(E_{\lambda_n})$, which implies $\mathrm{rk}(E_{\lambda})\ge 2n-1$. The equality holds if and only if $E_{\lambda}$ is a twist of the tangent bundle $E_{\lambda_2}$. Thus $\mathcal{L}_{\mathrm{ker}}$ cannot be a homogeneous bundle on $\mathbb{Q}^{2n-1}~(n\ge 4)$. The same reasoning applies to $\mathbb{Q}^{2n}$ for $n\ge 5$. We now consider the remaining cases.\\

For $\mathbb{Q}^{2n-1}$ with $n=3$, we have $\mathrm{rk}(E_{\lambda_3})=4<5=\mathrm{rk}(E_{\lambda_2})$, which implies $\mathrm{rk}(E_{\lambda})\ge 4$. The equality holds if and only if $E_{\lambda}$ is a twist of the spinor bundle $E_{\lambda_3}$. Let $E$ be an unsplit homogeneous bundle on $\mathbb{Q}^5$ of rank $5$. Then up to twisting a line bundle, $E$ is either the tangent bundle $E_{\lambda_2}$, or it has a filtration of length two with $E_{\lambda_3}$ as a factor. In the latter case, $E$ is isomorphic to the direct sum of $E_{\lambda_3}$ and a line bundle, as $E_{\lambda_3}$ is an ACM bundle. It follows that $\mathcal{L}_{\mathrm{ker}}$ cannot be a homogeneous bundle on this quadric. Likewise, $\mathcal{L}_{\mathrm{ker}}$ is not a homogeneous bundle on $\mathbb{Q}^3$.

For $\mathbb{Q}^{2n}$ with $n=4$, we have $\mathrm{rk}(E_{\lambda_4})=\mathrm{rk}(E_{\lambda_2})=8$, which implies $\mathrm{rk}(E_{\lambda})\ge 8$. The equality holds if and only if $E_{\lambda}$ is a twist of either the tangent bundle or a spinor bundle. Clearly, $\mathcal{L}_{\mathrm{ker}}$ is not a homogeneous bundle on $\mathbb{Q}^8$. 

For $\mathbb{Q}^6$, since $\mathrm{rk}(E_{\lambda_3})=4<6=\mathrm{rk}(E_{\lambda_2})$, we have $\mathrm{rk}(E_{\lambda})\ge 4$. The equality holds if and only if $E_{\lambda}$ is a twist of a spinor bundle. Let $E$ be an unsplit homogeneous bundle on $\mathbb{Q}^6$ of rank $6$. Then up to twist, 
$E$ is either the tangent bundle or has a filtration of length three with a spinor bundle as a factor. If the latter holds, then $E$ decomposes into a direct sum of a spinor bundle and two line bundles, as every spinor bundle is ACM. Consequently, $\mathcal{L}_{\mathrm{ker}}$ cannot be a homogeneous bundle on $\mathbb{Q}^6$ either.\\
%by irreducible homogeneous bundles. In the latter case, as $E$ is unsplit, one of the filtration factors must be a spinor bundle (up to dual and twist), and the remaining filtration factors must be line bundles or direct sums of line bundles. Since every spinor bundle is ACM, $E$ is isomorphic to the direct sum of a spinor bundle and line bundles. \\

It remains to consider the case of $\mathbb{Q}^4$. Note that $\mathbb{Q}^4 \cong G(2,4)=A_3/P_2$ and the spinor bundles $E_{\lambda_1}$ and $E_{\lambda_3}$ on $\mathbb{Q}^4$ are exactly the dual of the universal subbundle and quotient bundle on $G(2,4)$. Both are rank $2$ irreducible homogeneous bundles of minimal rank apart from line bundles. For $i=1,3$, $\mathrm{Sym}^kE_{\lambda_i}$ is an irreducible homogeneous bundle of rank $k+1$ with splitting type $(k,\ldots,2,1,0)$.
Let $E$ be an unsplit homogeneous $4$-bundle over $\mathbb{Q}^4$. If $E$ is irreducible, then it is a twist of $E_{\lambda_1+\lambda_3}$, $\mathrm{Sym}^3E_{\lambda_1}$ or $\mathrm{Sym}^3E_{\lambda_3}$. Now suppose $E$ is reducible. Its filtration has length $2$ or $3$.  For length $3$, a spinor bundle occurs as a filtration factor, and $E$ splits as the direct sum of this spinor bundle and two line bundles. For length $2$, $E$ fits into an exact sequence 
$
0\rightarrow  F\rightarrow E \rightarrow E/F \rightarrow 0$.

If $\mathrm{rk}(F)=3$ (resp. $1$), then $F$ (resp. $E/F$) is a twist of $\mathrm{Sym}^2E_{\lambda_i}$ for some $i~(i=1,3)$. By the Borel-Bott-Weil theorem, $H^1(\mathbb{Q}^4,\mathrm{Sym}^2E_{\lambda_i}(a))$ and $H^3(\mathbb{Q}^4,\mathrm{Sym}^2E_{\lambda_i}(a))$ vanish for all $a\in \mathbb{Z}$, which forces $E$ to decompose as the direct sum of a twist of  $\mathrm{Sym}^2E_{\lambda_i}$ and a line bundle. If $\mathrm{rk}(F)=2$, then we have $F\cong E_{\lambda_i}(a)$ and $E/F \cong E_{\lambda_j}(b)$. We now compute the extension group
\begin{align}\label{ext}
\mathrm{Ext}^1(E_{\lambda_j}(b),E_{\lambda_i}(a)) \cong H^1(E_{\lambda_i}\otimes E_{\lambda_j}(a-1-b)).
\end{align}

Suppose first that $i = j$. By the Littlewood-Richardson rule, $E_{\lambda_i}\otimes E_{\lambda_i}\cong \mathrm{Sym}^2E_{\lambda_i}\oplus \wedge^2E_{\lambda_i}$. 
Combined with the vanishing $H^1(\mathbb{Q}^4,\mathrm{Sym}^2E_{\lambda_i}(t))=0$ for any integer $t$ and the fact that every line bundle is ACM, we conclude that
the extension group (\ref{ext}) vanishes, so $E\cong E_{\lambda_i}(a)\oplus E_{\lambda_i}(b)$.
Now assume $i\neq j$. Using the Borel-Bott-Weil formula again, we get the extension group (\ref{ext}) fails to vanish if and only if $a-b = -1$.
Moreover, $h^1(E_{\lambda_i}\otimes E_{\lambda_j}(-2))=1$, which implies that non-trivial extensions of $E_{\lambda_j}(a+1)$ by $E_{\lambda_i}(a)$ are unique up to isomorphism. Notice that 
there exists a canonical non-trivial extension
$0\to E_i(a)\to \oh^{\oplus 4}_{\mathbb{Q}^4}(a+1)\to E_j(a+1)\to 0$. Consequently, if $i\neq j$, $E$ is isomorphic to either $E_{\lambda_i}(a)\oplus E_{\lambda_j}(b)$ or $\oh^{\oplus 4}_{\mathbb{Q}^4}(a+1)$ for integers $a,b$. We thus characterize all rank $4$ homogeneous bundles on $\mathbb{Q}^4$. Since $\mathcal{L}_{\mathrm{ker}}$ is stable with splitting type $(-1,0,0,0)$, it cannot be homogeneous. This concludes the proof.
\end{proof}

\begin{remark}
%In \cite{AZ}, Amram and Zhang proved that the threshold $k(\mathbb{Q}^N)<3N-8$ for homogeneity of uniform bundles on $\mathbb{Q}^N$ whenever $N\ge 4$.
Theorem \ref{non-homo} yields that $k(\mathbb{Q}^{2n-1})=2n-2$ and $2n-2\le k(\mathbb{Q}^{2n})\le 2n-1$, thereby providing a complete solution to Problem \ref{prob} for odd-dimensional quadrics. Moreover, for $n\ge 4$, the bundles $\mathcal{L}_{\mathrm{ker}}$ on $\mathbb{Q}^{2n-1}$ provide the first known examples of unsplit uniform bundles of minimal rank on generalized Grassmannians that fail to be homogeneous.
  
%Since every uniform bundle on $\mathbb{Q}^{2n-1}$ of rank at most $2n-2$ is homogeneous, Theorem \ref{non-homo} yields that the  threshold $k(\mathbb{Q}^{2n-1})$ for homogeneity of uniform bundles equals $2n-2$, thereby providing a complete solution to Problem \ref{prob} for odd-dimensional quadrics. Moreover, for $n\ge 4$, the bundles $\mathcal{L}_{\mathrm{ker}}$ on $\mathbb{Q}^{2n-1}$ provide the first known examples of unsplit uniform bundle of minimal ranks on generalized Grassmannians that fail to be homogeneous.
\end{remark}

	\section{The characterization of projective spaces}
	
Motivated by the bundles $\mathcal{L}_{\mathrm{ker}}$ on $\mathbb{Q}^n$, in this section we construct a family of vector bundles on arbitrary generalized Grassmannians not isomorphic to $\mathbb{P}^n$, and show that they are uniform but not homogeneous.
	
\begin{ex}[A class of uniform bundles $E$ on a generalized Grassmannian $X$]\label{ex:class-E}
	
Let $X$ be a generalized Grassmannian which is not isomorphic to a projective space. Then $X$ admits an embedding into $\mathbb{P}^N$ via the very ample line bundle $\mathcal{O}_X(1)$, and clearly $N > n:=\dim X$. Set $F_0:=T_{\mathbb{P}^N}(-1)|_X$.
Since $T_{\mathbb{P}^N}(-1)$ is generated by global sections, so is $F_0$.
We construct a sequence of quotient bundles $F_k$ of $F_0$ by iteration as follows.
Suppose $F_k$ has been constructed and $\mathrm{rk}(F_k)>n$. Since $F_k$ is globally generated, it has a general nowhere-vanishing global section $s_k$ by \cite[Corollary 5.5]{EH}. The section induces a monomorphism
\[
\mathcal{O}_X\xrightarrow{.s_k} F_k.
\]
Denote the quotient bundle by $F_{k+1}$.
Then $F_{k+1}$ is globally generated and has rank one less than that of $F_k$. The process can be repeated while  $\mathrm{rk}(F_{k+1})>n$. Since $\mathrm{rk}(F_k)=N$, after $N-n$ steps we obtain a globally generated quotient bundle $E:=F_{N-n}$ of rank $n$.
Since $X$ is a generalized Grassmannian, every line on $X$ is a line in $\mathbb{P}^N$. As $T_{\mathbb{P}^N}(-1)$ is uniform on $\mathbb{P}^N$ of splitting type $(1,0,\dots,0)$, the restriction $F_0$ is a uniform bundle on $X$ with the same
splitting type. Thus $E$ is uniform on $X$ of splitting type $(1,0,\dots,0)$.

Note that the above construction depends on the choices of the general sections $s_k$ at each step. Different choices may lead to non-isomorphic quotient bundles. Thus we obtain
a class of uniform bundles $E$.
\end{ex}

When $X$ is a smooth quadric, the family of bundles $E$ constructed in Example \ref{ex:class-E} coincides with the dual of the orthogonal bundles $\mathcal{L}_{\mathrm{ker}}$, which are known to be non-homogeneous. In this section we adopt a different approach to prove the non-homogeneity of the bundles $E$ constructed in Example \ref{ex:class-E}. 

\begin{theorem}\label{non-homo2}
Let $X=G/P_k$ be a generalized Grassmannian of dimension $n$. Suppose that $X$ is not isomorphic to $\mathbb{P}^n$. Then every uniform $n$-bundle on $X$ constructed in Example \ref{ex:class-E} is non‑homogeneous.
\end{theorem}
\begin{proof}
Let $X \hookrightarrow \mathbb{P}^N$ be the embedding induced by $\mathcal{O}_X(1)$ and let $E$ be a uniform bundle constructed in Example \ref{ex:class-E}. Then $E$ fits into the exact sequence
\[
0\rightarrow \oh_X^{\oplus N-n} \rightarrow  T_{\mathbb{P}^N}(-1)|_X\rightarrow E\rightarrow 0.
\]
Taking cohomology, we obtain $h^0(X,E)=h^0(T_{\mathbb{P}^N}(-1)|_X)-(N-n)=n+1$. Since $E$ is  globally generated, the evaluation morphism $\mathrm{ev}\colon H^0(X,E)\otimes \oh_X\to E$ is surjective and its kernel is the line bundle $\oh_X(-1)$. By the universal property of projective spaces, there exists a unique 
 morphism $\Phi\colon X\to  Y:=\mathbb{P}(H^0(X,E))$ such that $\Phi^*\oh_Y(-1)=\mathcal \oh_X(-1)$. Suppose, for contradiction, that $E$ is homogeneous. By \cite[Lemma 9.8]{ott}, $H^0(X,E)$ is a $G$-module, so the evaluation map $\mathrm{ev}$ is $G$-equivariant, which implies that $\Phi$ is also $G$-equivariant. Since $\Phi$ is non-constant and $X$ has Picard number $1$, $\Phi$ must be finite. Furthermore, as $\dim X=\dim Y$,
%the dimensions of the domain and the codomain are equal, 
$\Phi$ must be surjective. Since $\Phi$ is surjective and $G$-equivariant, the transitivity of $G$ on $X$ implies transitivity on $Y$. So $Y$ is a homogeneous space $G/Q$ for some parabolic subgroup $Q$. %, where $Q$ denotes the stabilizer in $G$ of some point $y\in Y$. 
 Thus $\Phi$ is a $G$-equivariant morphism from $G/P_k$ to $G/Q$, and evidently $P_k\subset Q$. Since $P_k$ is a maximal parabolic subgroup, we have $P_k=Q$. Consequently,
$X= G/P_k = G/Q \cong \mathbb{P}^n$, which contradicts the hypothesis. This completes the proof.
\end{proof}
	
	As a consequence, we obtain the bound 
	$k(X)\le \dim X-1$ for every generalized Grassmannian $X$ not isomorphic to a projective space. This gives Theorem \ref{main2}.\\

It is worth noting that $k(\mathbb{P}^n)\ge n+1$, while for generalized Grassmannians $X\not\cong\mathbb P^n$ we have $k(X)\le \dim X-1$. In view of this numerical discrepancy, it is natural to ask whether projective spaces can be characterized in terms of uniform bundles. The following theorem provides an affirmative answer.
	
%It is well known that for arbitrary generalized Grassmannian $X$, the tangent bundle $T_X$ is uniform with splitting type $(2,1^{\times p},0,\ldots,0)$, where $p=c_1(T_X)-2$. In \cite{OCW}, Occhetta, Sol\'{a} Conde and Watanabe characterize generalized Grassmannians among Fano manifolds of Picard number one via stronger uniformity conditions on a family of minimal rational curves. In this section, we give a simple characterization of projective spaces among all generalized Grassmannians from the viewpoint of uniform bundles.
	
	\begin{theorem}
	Let $X$ be a generalized Grassmannian of dimension $n\ge 2$. Then $X\cong\mathbb{P}^n$ if and only if $T_X$ is the unique unsplit uniform bundle on $X$ of minimal rank, up to dual and twist.
	\end{theorem} 
	\begin{proof}
The only if part follows from the classification of uniform bundles of rank at most $n$ on $\mathbb{P}^n$. For the if part, recall first that $T_X$ is an unsplit uniform bundle of splitting type $(2,1^{\times p},0,\ldots,0)$, where $p=\deg T_X-2$. The hypothesis that $T_X$ is the unique unsplit uniform bundle on $X$ of minimal rank implies that the minimal rank for unsplit uniform bundles equals $n(=\dim X)$. \\

Suppose $X$ is not isomorphic to $\mathbb{P}^n$. Let $E$ be a uniform bundle constructed in Example \ref{ex:class-E}. By construction, $E$ has the same Chern classes as $T_{\mathbb{P}^N}(-1)|_X$, hence
\[
c_i(E)=c_i(T_{\mathbb{P}^N}(-1)|_X)=H^i,\quad 1\le i\le n,
\]
 where $H$ is the ample generator of $\mathrm{Pic}(X)$. Since the splitting type of $E$ is $(1,0,\dots,0)$, comparing these Chern classes with those of $\oh_X(1)\bigoplus \oh_X^{\oplus (n-1)}$ shows that $E$ does not split. Thus $E$ is an unsplit uniform bundle on $X$ of minimal rank.
 
 Assume that for some $a,b\in\mathbb Z$, $E(a)$ or $E^*(b)$ is isomorphic to $T_X$. Then the unique $\mathcal{O}(2)$-factor would appear in the splitting type of $E(a)$ or $E^*(b)$, which forces $T_X \cong E(1)$. So the index of $X$ is $n + 1$, as $c_1(T_X)=c_1(E(1))=(n + 1)H$. By the Kobayashi-Ochiai theorem, $X$ is isomorphic to $\mathbb{P}^n$, contradicting the hypothesis. 
 Therefore, if $X$ is not isomorphic to $\mathbb{P}^n$, then up to dual and twist, $T_X$ is not the unique unsplit uniform bundle on $X$ of
 minimal rank.   
 This completes the proof of sufficiency.
\iffalse
Suppose $X$ is not isomorphic to a projective space. Then $X$ admits an embedding into $\mathbb{P}^N$ via the very ample line bundle $\mathcal{O}_X(1)$, and clearly $N > n$. Note that $T_{\mathbb{P}^N}(-1)$ is generated by global sections, hence so is $T_{\mathbb{P}^N}(-1)|_X$. Since $\mathrm{rk}(T_{\mathbb{P}^N}(-1)|_X)>n$, \cite[Corollary 5.5]{EH} tells us that $T_{\mathbb{P}^N}(-1)|_X$ contains $\mathcal{O}_X$ as a subbundle, yielding a quotient bundle $M_1$ of $T_{\mathbb{P}^N}(-1)|_X$. Meanwhile, $X$ is a generalized Grassmannian, so every line on $X$ is a line in $\mathbb{P}^N$. As $T_{\mathbb{P}^N}(-1)$ is uniform on $\mathbb{P}^N$, its restriction $T_{\mathbb{P}^N}(-1)|_X$ is a uniform bundle on $X$ with splitting type $(1,0,\dots,0)$, and hence $M_1$ is uniform of rank $N-1$ with splitting type $(1,0,\dots,0)$. Moreover, it is readily seen that $M_1$ is also globally generated. If $\mathrm{rk}(M_1)>n$, we apply \cite[Corollary 5.5]{EH} again to $M_1$, obtaining another $\mathcal{O}_X$-subbundle and denoting its quotient by $M_2$.
Iterating this procedure, we eventually obtain a quotient bundle %$E:=M_{N-\dim X}$ 
$E$ of $T_{\mathbb{P}^N}(-1)|_X$ with $\mathrm{rk}(E)=n$, which is a uniform bundle on $X$ with splitting type $(1,0,\dots,0)$.\\
\textbf{Claim:} $E$ does not split and no twist of $E$ or its dual can be isomorphic to $T_X$.
\fi

\end{proof}
	%\textbf{Acknowledgements:} The authors would like to thank  Rong Du for fruitful discussions and valuable suggestions on this article.
	
	%introducing this problem to us and for their very helpful discussions during the last several months. We also thank Peng Ren for  valuable discussions and suggestions. The authors are  grateful to Gianluca Occhetta for his generosity of sharing ideas with us. His  suggestions improve the first version of this article hugely.

	\bibliography{ref}
	\bibliographystyle{plain}
	
	\end{document}